\documentclass{article}
\usepackage{amsmath}
\usepackage{amsthm}
\usepackage{cite}
\usepackage{amscd}
\usepackage{authblk}
\usepackage[all]{xy}
\theoremstyle{definition}
\newtheorem{thm}{Theorem}[section]
\newtheorem{definition}[thm]{Definition}
\newtheorem{prop}[thm]{Proposition}
\newtheorem{lem}[thm]{Lemma}
\newtheorem{rem}[thm]{Remark}
\newtheorem{cor}[thm]{Corollary}
\newtheorem{ex}[thm]{Example}

\newtheorem{ques}[thm]{Question}
\newtheorem*{ack}{Acknowledgement}
\newtheorem*{nota}{Notation}
\numberwithin{equation}{section}
\usepackage{amssymb}
\newcommand{\Z}{\mathbb{Z}}
\newcommand{\Q}{\mathbb{Q}}
\newcommand{\Spec}{\operatorname{Spec}}

\newcommand{\Hom}{{\rm Hom}}

\newcommand{\Rep}{{\rm Rep}_k}
\newcommand{\Vect}{{\rm Vec}_k}

\newcommand{\Dim}{{\rm dim}}

\newcommand{\GL}{{\rm GL}}
\newcommand{\SL}{{\rm SL}}

\newcommand{\Ker}{{\rm Ker}}

\newcommand{\calO}{\mathcal{O}}

\newcommand{\Def}{\overset{{\rm def}}{=}}
\newcommand{\ab}{{\rm ab}}
\newcommand{\et}{\text{{\rm \'et}}}
\newcommand{\lr}{{\rm lin.red}}
\newcommand{\loc}{{\rm loc}}
\newcommand{\fpqc}{{\rm fpqc}}
\newcommand{\Diag}{{\rm Diag}}
\newcommand{\X}{\mathbb{X}}
\newcommand{\G}{\mathbb{G}}
\newcommand{\F}{\mathbb{F}}
\newcommand{\A}{\mathbb{A}}
\newcommand{\C}{\mathbb{C}}
\newcommand{\Str}{{\rm Strat}}
\newcommand{\str}{\text{$1$-strat}}
\newcommand{\unit}{\mathbb{I}}

\title{On a purely inseparable analogue of the Abhyankar Conjecture for affine curves}
\author{Shusuke Otabe
\footnote{Mathematical Institute, Graduate School of Science, Tohoku University, 6-3 Aramakiaza, Aoba, Sendai, Miyagi 980-8578, Japan;
E-mail: {shusuke.otabe.r1@dc.tohoku.ac.jp}
}}
\affil{Tohoku University}
\date{\vspace{-12mm}}

\begin{document}

\maketitle

\begin{abstract}
Let $U$ be an affine smooth curve defined over an algebraically closed field of positive characteristic. The Abhyankar Conjecture (proved by Raynaud and Harbater in 1994) describes the set of  finite quotients of Grothendieck's \'etale fundamental group $\pi_1^{\text{\'et}}(U)$. In this paper, we consider a purely inseparable analogue of this problem, formulated in terms of Nori's profinite fundamental group scheme $\pi^N(U)$, and give a partial answer to it.
\end{abstract}

\thispagestyle{empty}

\section{Introduction}

\subsection{Nori's fundamental group scheme}
Let $X$ be an algebraic variety over a field $k$. In \cite{gr71}, Grothendieck defined the \textit{\'etale fundamental group} $\pi_1^{\et}(X)$ as a generalization of the fundamental group of a topological space. It is the smallest profinite group classifying $G$-coverings over $X$, where $G$ is a finite group. In the case where $k$ is an  algebraically closed field of characteristic zero with $k\hookrightarrow\C$, it is known that $\pi_1^{\et}(X)$ is isomorphic to the profinite completion of the topological one $\pi_1^{{\rm top}}(X(\C))$. For example, as the complex line $\A^1(\C)$ is simply connected, we have $\pi_1^{\et}(\A_{\C}^1)=0$. On the other hand, in the case where $k$ is of positive characteristic $p>0$, the situation is quite different. This can be seen even in the case where $X=\A^1_{k}$ is the affine line. Indeed, it is known that  $\Dim_{\F_p}\Hom(\pi^{\et}_1(\A_k^1),\F_p)=\infty$. On the other hand, if $X=U$ is a smooth affine curve, the \textit{Abhyankar Conjecture}~\cite{ab57}, proved by Raynaud and Harbater~\cite{ha94}\cite{ra94}, gives us  another estimate of the difference between $\pi^{\et}_1(U)$ and the topological one of a Riemann surface of the same type. The conjecture describes the set
\begin{equation}\label{eq:et_A}
\pi_A^{\et}(U)\Def\{\text{finite quotients of  $\pi^{\et}_1(U)$}\}
\end{equation} 
for any smooth affine curve $U$. Here, more precisely, $\pi^{\et}_A(U)$ is the set of isomorphism classes of fintie groups which appears as a finite quotient of $\pi_1^{\et}(U)$. For example, it says for any integer $n>0$ and $r>0$, there exists a surjective homomorphism $\pi^{\et}_1(\A^1_k)\twoheadrightarrow\SL_n(\F_{p^r})$~
(cf.~\cite{ka85}\cite{no94}\cite{se92}.)

In \cite{no76}\cite{no82}, Nori defined the \textit{fundamental group scheme} $\pi^N(X)$ as a generalization of Grothendieck's \'etale geometric fundamental group of an algebraic variety $X$ over a field $k$~(cf.~\cite{gr71}). It is a profinite $k$-group scheme classifying $G$-torsors over $X$ with $G$ a finite $k$-group scheme. In \cite{no76}, he first constructed it under the assumption that $X$ is proper over $k$ by using the theory of Tannakian categories. On the other hand, in \cite{no82}, he also proved that such a profinite one $\pi^N(X)$ exists also for arbitrary (not necessarily proper) $X$ without relying on a Tannakian construction~(cf.~Section \ref{sec:def fund grp sch}). 
Note that, if $k$ is an algebraically closed field, then Grothendieck's \'etale fundamental group $\pi_1^{\et}(X)$ classifies all finite \'etale torsors over $X$ and is universal for this property. Hence, there exists a homomorphism of $\pi^N(X)$ into the pro-constant $k$-group scheme, denoted by $\pi^{\et}(X)$, associated with $\pi_1^{\et}(X)$. In fact $\pi^{\et}(X)$ gives the maximal pro-\'etale quotient of $\pi^N(X)$, called the \textit{\'etale fundamental group scheme}. In the case where $k$ is of characteristic zero, then the surjective homomorphism $\pi^N(X)\twoheadrightarrow\pi^{\et}(X)$ is in fact an isomorphism. This is valid because under the assumption that $k$ is an algebraically closed field of characteristic zero, any fintie $G$-torsor $P\to X$ with $G$ a finite $k$-group scheme $G$ is nothing but a $G(k)$-covering over $X$. On the other hand, in the case where $k$ is of positive characteristic $p>0$, then $\pi^N(X)$ is strictly larger than $\pi^{\et}(X)$, in general. Indeed, finite \textit{local} (``purely inseparable'') torsors make a contribution to occur the difference between these fundamental group schemes. Here a finite $k$-group scheme $G$ is said to be \textit{local} if it is connected, i.e., if $G^0$ denotes the connected component of the identity $1\in G$, then $G^0=G$. For example, $G=\alpha_p$, or $\mu_p$.

\subsection{Main results}

In the present paper, we will attempt to estimate the difference between $\pi^N(U)$ and $\pi^{\et}(U)$ for a smooth affine curve $U$ defined over an algebraically closed field $k$ of positive characteristic $p>0$ from the viewpoint of the Inverse Galois Problem. We will study a purely inseparable analogue of the Abhyankar Conjecture for affine curves $U$~(cf.~\cite{ab57}\cite{ab92}; see also Section \ref{sec:int-classical}), i.e., we will try to describe the set
\begin{equation*}
\pi_A^{\loc}(U)\Def\{\text{finite local quotients of  $\pi^{N}(U)$}\}.
\end{equation*}
More precisely, $\pi^{\loc}_A(U)$ is the set of isomorphism classes of fintie local $k$-group schemes  which appears as a finite quotient of $\pi^{N}(U)$.

Now let us explain the contents of the present paper. In Section \ref{sec:preliminaries}, we will briefly review the definition of Nori's profinite fundamental group scheme and the maximal linearly reductive quotient of it. 
In Section \ref{sec:a purely insep}, we will see the maximal local linearly reductive quotient of $\pi^N(U)$ provides a necessarily condition for a finite local $k$-group scheme $G$ to belong to the set $\pi^{\loc}_A(U)$~(cf.~Proposition \ref{prop:necessarily cond}). Now let us explain this. Let $X$ be a smooth compactification of $U$. 
Let $n\Def\#(X\setminus U)$. Note that the affineness assumption of $U$ implies $n>0$. Let $\gamma$ be the $p$-rank of the Jacobian variety of $X$, i.e., $\gamma=\Dim_{\F_p}{\rm Pic}^0_X[p](k)$. We will see that  for any finite local $k$-group scheme $G$, if $G\in\pi_A^{\loc}(U)$, then the character group $\X(G)\Def\Hom(G,\G_m)$ must be embeddable as a subgroup into $(\Q_p/\Z_p)^{\oplus \gamma+n-1}$. Then we can ask whether or not the converse is true~(cf.~Question \ref{ques:purely insep}):

\begin{ques}\label{ques:int-purely insep}
Let $U$ be a smooth affine curve and $G$ a finite local $k$-group scheme.  If there exists an injective homomorphism $\X(G)\hookrightarrow (\Q_p/\Z_p)^{\oplus \gamma+n-1}$, then does $G$ belong to the set $\pi^{\loc}_A(U)$?
\end{ques}

The main purpose of the present paper is to give a partial affirmative answer to Question \ref{ques:int-purely insep}. The main result is the following~(cf.~Proposition \ref{prop:nilpotent};~Corollary \ref{cor:bertini};~Corollary \ref{cor:GL_2}):

\begin{thm}\label{thm:int-main}
(1) For any smooth affine curve $U$ over $k$ and any finite local nilpotent $k$-group scheme $G$, if there exists an injective homomorphism $\X(G)\hookrightarrow (\Q_p/\Z_p)^{\oplus \gamma+n-1}$, then there exists a surjective homomorphism $\pi^N(U)\twoheadrightarrow G$.\\
(2) Let $\Sigma$ be a semi-simple simply connected algebraic group over $k$. Then for any integer $r>0$, there exists a surjective homomorphism
\begin{equation*}
\pi^{N}(\A_k^1)\twoheadrightarrow\Sigma_{(r)}
\end{equation*}
of $\pi^N(\A_k^1)$ onto the $r$-th Frobenius kernel $\Sigma_{(r)}\Def\Ker(F^{(r)}:\Sigma^{(-r)}\to\Sigma)$.\\
(3) Assume $p=2$. Then for any integer $r>0$, there exists a surjective homomorphism
\begin{equation*}
\pi^{N}(\G_m)\twoheadrightarrow\GL_{2(r)}
\end{equation*}
of $\pi^N(\G_m)$ onto the $r$-th Frobenius kernel $\GL_{2(r)}$ of $\GL_2$.
\end{thm}

Here, for each integer $r>0$, $\Sigma^{(-r)}$ denotes the $(-r)$-th Frobenius twist of $\Sigma$ and $F^{(r)}:\Sigma^{(-r)}\to\Sigma$ is the $r$-th relative Frobenius morphism, which is a homomorphism of algebraic groups. Furthermore, its kernel $\Sigma_{(r)}=\Ker(F^{(r)})$ is a finite local $k$-group scheme.

\begin{rem}
(1) If $\Sigma$ is a semi-simple simply connected algebraic group over $k$, it turns out that $\X(\Sigma_{(r)})=1$~(cf.~Remark \ref{rem:purely insep}(3)). Therefore, Theorem \ref{thm:int-main}(2) gives an affirmative answer to Question \ref{ques:int-purely insep} for the affine line $\A_k^1$ and for the Frobenius kernels   $\Sigma_{(r)}~(r>0)$ of a semi-simple simply connected algebraic group $\Sigma$.\\
(2) Since $\X(\GL_{2(r)})=\Z/2^r\Z\subset\Q_2/\Z_2$ if $p=2$, Theorem \ref{thm:int-main}(3) gives an affirmative answer to Question \ref{ques:int-purely insep} for the multiplicative group $U=\G_m$ and for the  Frobenius kernels $\GL_{2(r)}~(r>0)$ of $\GL_2$ in the case where $k$ is of characteristic $p=2$.\\
(3) Note that each homomorphism $\pi^N(U)\to G$ of $\pi^N(U)$ into a finite $k$-group scheme $G$ corresponds bijectively to a (fpqc) $G$-torsor $P\to U$~(cf.~Proposition \ref{prop:univ}). To prove Theorem \ref{thm:int-main}(2), we will show there exists a $k$-morphism $f:\A_k^1\to\Sigma$ such that the resulting $\Sigma_{(r)}$-torsor $f^*\Sigma^{(-r)}\to\A_k^1$ realizes a surjective homomorphism $\pi^N(\A_k^1)\twoheadrightarrow\Sigma_{(r)}$. To prove the existence of such a morphism $f$, we will first reduce the problem to the case where $r=1$~(cf.~Lemma \ref{lem:any ht}). Next we will deduce the existence of such an $f$ from a Bertini type theorem for height one torsors~(cf.~Theorem \ref{thm:bertini}).\\
(4) In particular, Theorem \ref{thm:int-main}~(1) (or Theorem \ref{thm:int-main}(2)) implies that if $p=2$, then $\SL_{2(1)}$ appears as a finite quotient of $\pi^N(\A_k^1)$~(cf.~Example \ref{ex:SL_2}). On the other hand, it turns out that any  $\SL_{2(1)}$-torsor must be of the form $f^*\SL_2^{(-1)}\to\A_k^1$ for some $k$-morphism $f:\A_k^1\to\SL_2$. In this case, we can explicitly describe  the subset of ${\rm Mor}_k(\A_k^1,\SL_2)$ consisting of $k$-morphisms $f$ such that the torsor $f^*\SL_2^{(-1)}$ realizes a surjective homomorphism $\pi^N(\A_k^1)\twoheadrightarrow\SL_{2(1)}$~
(cf.~Corollary \ref{cor:SL_2}).
\end{rem}

\begin{rem}
Considering recent developments of the theory of  `\textit{tame stacks}' (cf.~\cite{aov08}\cite{ma12}\cite{gi12}), the author expects that the notion of linealy reductiveness might provide us with a good analogy between the answer of Question \ref{ques:int-purely insep} and the Abhyankar Conjecture~(cf.~Theorem \ref{thm:abhy}) and that Question \ref{ques:int-purely insep} might be affirmative for any smooth affine curve $U$ and for any finite local $k$-group scheme $G$. See also Remark \ref{rem:int-tame torsors}.
\end{rem}

\subsection{The Abhyankar Conjecture}\label{sec:int-classical}

All the ideas of our arguments in the present paper come from Serre's work~\cite{se90} (The method of embedding problems) or Nori's one \cite{ka85}\cite{no94} in the sequel of the Abhyankar Conjecture for the affine line. Hence, we would like to briefly review on the  conjecture.

First we will recall the precise  statement of it (in a weak form). Let $k$ be an algebraically closed field of positive characteristic $p>0$ and $X$ a smooth projective curve over $k$ of genus $g>0$. Let $\emptyset\neq U\subset X$ be a nonempty open subset with $n\Def\#(X\setminus U)\ge 0$. Let $\Gamma_{g,n}$ be the group defined by
\begin{equation*}
\Gamma_{g,n}\Def\langle a_1,b_1,\dots, a_g,b_g,\gamma_1,\dots,\gamma_n\,|\,\prod_{i=1}^{g}[a_i,b_i]\gamma_1\cdots\gamma_n=1\rangle
\end{equation*}
and $\widehat{\Gamma}_{g,n}$ by its profinite completion. Note that if $n>0$, then $\Gamma_{g,n}$ is a free group $F_{2g+n-1}$ of rank $2g+n-1$.
A classical result due to Grothendieck \cite{gr71} implies:
\begin{equation*}
\pi^{\et}_1(U)^{(p')}\simeq\widehat{\Gamma}_{g,n}^{(p')}.
\end{equation*}
Here $(-)^{(p')}$ means the maximal pro-prime to $p$ quotient.   In particular, if $n>0$, i.e., $U$ is affine, then $\pi_1^{\et}(U)^{(p')}\simeq\widehat{F}_{2g+n-1}^{(p')}$. Hence, in this case, if a finite group $G$ appears as a finite quotient of $\pi^{\et}_1(U)$, then $G^{(p')}$ must be  generated by at most $2g+n-1$ elements. Here, the group $G^{(p')}$ is the quotient of $G$ by the subgroup $p(G)$ generated by all the $p$-Sylow subgroups of $G$. The Abhyankar Conjecture claims that the converse is also true:

\begin{thm}\label{thm:abhy}(The Abhyankar Conjecture in a weak form~\cite{ra94}\cite{ha94})
Assume that $U$ is affine. Let $G$ be an arbitrary finite group. Then $G\in\pi^{\et}_A(U)$~(cf.~(\ref{eq:et_A})) if and only if $G^{(p')}$ can be  generated by at most $2g+n-1$ elements.
\end{thm}

Raynaud proved the threorem for the case where $U=\A^1_k$~\cite{ra94}. Soon after, Harbater obtained for the general case~\cite{ha94}.

\begin{rem}(1)~Theorem \ref{thm:abhy} implies that $\pi^{\et}_A(U)$ is determined by the topological one $\pi_1^{\rm top}(U_0)$ of a Riemann surface $U_0$ of the same type $(g,n)$ as $U$. An affirmative answer to our question  (Question \ref{ques:int-purely insep}) says that the set of finite local quotients of $\pi^N(U)$ might be determined by the \'etale one $\pi_1^{\et}(U)$. Note that the number $\gamma+n-1$ can be reconstructed from $\pi_1^{\et}(U)$ group-theoretically (cf. \cite{ta99}).\\
(2)~The assumption that $U$ is affine is essential. This is because if $U=X$ is projective, then $\pi^{\et}_1(X)$ is topologically finitely generated and the isomorphism class of $\pi^{\et}_1(X)$ can be completely determined by the set $\pi^{\et}_A(X)$ of finite quotients of it~(cf.~\cite[Proposition 5.4]{fj08}). On the other hand, it is known that $\pi_1^{\et}(X)$ itself has much information about the moduli of $X$~(cf.~\cite{ta99}\cite{ta04}) if the genus of $X$ is $g\ge 2$. Therefore, such a simple   answer as above (Theorem \ref{thm:abhy}) cannot be expected and the corresponding problem is much more challenging~(cf.~\cite{ps00}). 
\end{rem}

In the particular case $U=\A^1_k$ is the affine line, the Abhyankar Conjecture states that a finite group $G$ belongs to the set $\pi_A^{\et}(\A^1_k)$ if and only if $G^{(p')}=1$. The latter condition is equivalent to the one that $G$ can be generated by $p$-Sylow subgroups of it, and such a group is called a \textit{quasi-$p$-group}. Obviously, any $p$-group is a quasi-$p$-group. In \cite{ab57}\cite{ab92}, Abhyankar found an explicit equation defining a finite \'etale Galois covering over $\A^1_k$ whose corresponding homomorphism $\pi^{\et}_1(\A^1_k)\to G$ is surjective for various nontrivial quasi-$p$-groups $G$.  In \cite{se90}, Serre appoarched the conjecture for the affine line $\A^1_k$ by the method of embedding problems. As a result, he proved that the conjecture is ture for any solvable quasi-$p$-group. This result gives a first reduction step in the proof due to Raynaud~\cite{ra94}. On the other hand, although his result was not used in Raynaud's proof, Nori provided many examples of quasi-$p$ groups arrearing as a finite quotient of $\pi^{\et}_1(\A_k^1)$, which gave another evidence of the Abhyankar conjetcure for $\A_k^1$~(cf.~\cite{ka85}\cite{no94}). Let $\Sigma$ be a semi-simple simply connected algebraic group (for example, $\Sigma=\SL_n$) over $\F_{p^r}$. Let $L:\Sigma\to\Sigma;~g\mapsto g^{-1}F^r(g)$ be the Lang map, where $F$ is the absolute Frobenius morphism of $\Sigma$. Note that $L$ gives a Galois covering with group $\Sigma(\F_{p^r})$. Nori showed the existence of a closed immersion $\iota:\A^1_{\F_{p^r}}\hookrightarrow\Sigma$ such that $\iota^*\Sigma$ is geometrically-connected, whence the Galois covering $(\iota^*\Sigma)_k\to\A^1_k$ corresponds to a surjective homomorphism $\pi^{\et}_1(\A^1_k)\twoheadrightarrow\Sigma(\F_{p^r})$.
For a brief survey of all the above results, see, for example, \cite[Section 3]{ha14}.

\begin{rem}(1) Our result for finite local nilpotent group schemes (Theorem \ref{thm:int-main}(1)~(cf.~Proposition \ref{prop:nilpotent})) is motivated by Serre's result on solvable quasi-$p$-groups \cite{se90}. However, we can extend his method only in  the nilpotent case.\\
(2) Theorem \ref{thm:int-main}(2) can be considered as a purely inseparable analogue of Nori's result \cite{ka85}\cite{no94}. To prove this, we will rely on a Bertini type theorem (Theorem \ref{thm:bertini}). Hence, we cannot give an explicit equation defining a saturated $\Sigma_{(1)}$-torsor over $\A_k^1$ for general $\Sigma$. In the case where $p=2$ and $\Sigma=\SL_2$, we can clarify this situation in not a  conceptual but an explicit way.\\
(3) In the classical conjecture, in particular, in the proof due to Raynaud or Harbater, the rigid  analytic or formal patching method and the theory of stable curves provide us with strong tools to solve the problem. The author is not sure if these methods can be applicable in our situation.\\
\end{rem}

\begin{rem}\label{rem:int-tame torsors}
The full statement of the classical Abhyankar Conjecture (proved by Harbater) states that as a covering realizing a finite quotient $G\in\pi_A^{\et}(U)$, one can take a $G$-covering tamely ramified except for one point $x_0\in X\setminus U$~(cf.~\cite[Conjecture 1.2]{ha94}). So, it might be natural to ask for an analogous problem.
To formulate it, we need the notion of \textit{tamely ramified torsors}, a similar notion to tamely remified coverings. 

The notion of \textit{tameness} of an action of a group scheme $G$ on a scheme $X$ was introduced by Chinburg-Erez-Pappas-Taylor~\cite{cept96}. On the other hand, Abramovich-Olsson-Vistoli gave another formulation in terms of \textit{tame stacks} (cf.~\cite{aov08}). A relation between these two works has been studied by Marques~\cite{ma12}. Moreover, in \cite{bo09}, Borne defined the fundametal group scheme which classifies tamely ramified torsors by using the  Tannakian category of parabolic bundles.
To obtain a good analogue of the  notion of tame coverings, one needs  to consider an extension of a $G$-torsor $P$ over an open subset $U$ of $X$ with $D=X\setminus U$ a normal crossing divisor to an $X$-scheme $Q$ together with an action of $G$. In \cite{gi12}, Gillibert discussed on this point. Recently, in \cite{za16}, Zalamansky formulated a ramification theory in  purely inseparable setting in terms of ramification divisors. The author is not sure if there exists any relations between Zalamansky's formulation and the previous ones.

 In view of the above recent developments of the theory of  tamely ramified torsors, one can formulate a naive analogue of the strong Abhyankar Conjecture \cite[Conjecture 1.2]{ha94} in an obvious way. However, the author has no evidence of it. So, in the present paper, we will concentrate an analogue of the weak Abhyankar Conjecture.
\end{rem}

\begin{ack} The author thanks Professor Takao Yamazaki for his encouragements, for clarifying discussions and for giving him helpful advices. The author thanks Professor Takuya Yamauchi for having conversations about this problem and for encouraging him. The author thanks Doctor Fuetaro Yobuko for answering his several questions. The author thanks Professor Seidai Yasuda for pointing out a mistake to him at the coference ``Regulators in Niseko 2017". The author thanks the reviewer for giving him detailed and valuable comments on the first version of the paper.
The author is supported by JSPS, Grant-in-Aid for Scientific Research for JSPS fellows (16J02171).
\end{ack}

\begin{nota}
In this paper, $k$ always means a perfect field; an \textit{algebraic variety} over $k$ means a geometrically-connected and reduced scheme separated of finite type over $k$; a \textit{curve} over $k$ is an algebraic variety over $k$ of dimension one; an \textit{algebraic group} over $k$ means a group object of the category of algberaic varieties over $k$. Note that, automatically, any algebraic group over $k$ is smooth.

Let $k$ be a perfect field of positive characteristic $p>0$ and $X$ an algebraic variety over $k$. We denote by $F_X$, or simply $F$, the absolute Frobenius morphism $F=F_X:X\to X$. For each integer $n\in\Z$, we denote by $X^{(n)}$ the $n$-th \textit{Frobenius twist} of $X$:
\begin{equation*}
\begin{xy}
\xymatrix{\ar@{}[rd]|{\square}
X^{(n)}\ar[r]\ar[d]& X\ar[d]\\
\Spec k\ar[r]_{\simeq}^{F^n}&\Spec k.
}
\end{xy}
\end{equation*}
If $n>0$, the morphism $F^{n}:X\to X$ then factors uniquely through $X^{(n)}$. The resulting morphism, denoted by $F^{(n)}:X\to X^{(n)}$, is the $n$-th relative Frobenius morphism.

We denote by $\Vect$ the category of finite dimensional vector spaces over $k$. For an affine $k$-group scheme $G$, we denote by $\Rep(G)$ the category of finite dimensional left $k$-linear representations of $G$. For each $(V,\rho)\in\Rep(G)$, we denote by $V^G$ the $G$-invariant subspace of $G$, i.e., $V^G=\{v\in V\,|\,\rho(v)=v\otimes 1\}$.
\end{nota}

%%%%%%  Preliminaries   %%%%%%%%%%%%%

\section{Fundamental group scheme}\label{sec:preliminaries}

\subsection{Profinite fundamental group scheme}\label{sec:def fund grp sch}

In this subsection, we will briefly recall the definition of \textit{Nori's profinite fundamental group scheme}~\cite{no76}\cite{no82}.

Let $X$ be an algebraic variety over $k$ together with a rational point $x\in X(k)$. We define the category $N(X,x)$ as follows. The objects of $N(X,x)$ are all the triples $(P,G,p)$ where
\begin{itemize}
\item $G$:~a finite $k$-group scheme;
\item an (fpqc) $G$-torsor $\pi:P\to X$;
\item $p\in P(k)$:~a rational point with $\pi(p)=x$.
\end{itemize}
Let $(P,G,p),(Q,H,q)\in N(X,x)$ be arbitrary two objects. Then a morphism $(P,G,p)\to (Q,H,q)$ is a pair $(f,\phi)$ of an $X$-morphism $f:P\to Q$ and a $k$-homomorphism $\phi:G\to H$ making the following diagram commute
\begin{equation*}
\begin{xy}
\xymatrix{
P\times G\ar[r]\ar[d]_{(f,\phi)}& P\ar[d]^{f}\\
Q\times H\ar[r]& Q.
}
\end{xy}
\end{equation*}
Here, the above two horizontal morphisms are the ones defining the actions of torsors. By these objects and morphisms, $N(X,x)$ becomes a category. In \cite{no82}, Nori proved that the category $N(X,x)$ is a cofiltered category and, in particular that the projective limit
\begin{equation*}
\varprojlim_{(P,G,p)\in N(X,x)}(P,G,p)
\end{equation*}
exists.

\begin{definition}
The projective limit of underlying group schemes
\begin{equation*}
\pi^N(X,x)\Def\varprojlim_{(P,G,p)\in N(X,x)}G
\end{equation*}
is called the \textit{profinite fundamental group scheme}, or shortly the \textit{fundamental group scheme}, of $(X,x)$.
\end{definition}

From the definition, the following is immediate:

\begin{prop}\label{prop:univ}
The fundamental group scheme $\pi^N(X,x)$ is a profinite $k$-group scheme such that for any finite $k$-group scheme $G$,  the map
\begin{equation*}
\begin{aligned}
\Hom_k(\pi^N(X,x),G)~&\to~{\rm Tors}(G,(X,x))\\
\phi\quad ~&\mapsto~(X_x^N,x^N)\times^{\pi^{N}(X,x)}{}^{\phi}G
\end{aligned}
\end{equation*}
is bijective. Here
\begin{equation*}
(X_x^N,x^N)\Def\varprojlim_{(P,G,p)\in N(X,x)} (P,p)
\end{equation*}
and ${\rm Tors}(G,(X,x))$ is the set of isomorphism classes of pointed $G$-torsors over $(X,x)$. Moreover,  the  torsor $(X_x^N,x^N)\times^{\pi^{N}(X,x)}{}^{\phi}G$ associated with a homomorphism $\phi$ is defined as the quotient of the product $(X^N_x,x^N)\times G$ by the diagonal action of $\pi^N(X,x)$:
\begin{equation*}
(x,g)\cdot \gamma\Def(xg^{-1},g\phi(\gamma))
\end{equation*}
for $(x,g)\in X^N_x\times G$ and $\gamma\in\pi^N(X,x)$.
\end{prop}

\begin{definition}\label{definition:saturated}
A $G$-torsor $(P,p)\to (X,x)$ is said to be \textit{saturated} if the corresponding homomorphism $\pi^N(X,x)\to G$ is surjective.
\end{definition}

\begin{rem}
If $k$ is an algebraically closed, then for any finite \'etale $k$-group scheme $G$ is the constant group scheme associated with the finite group $G(k)$ and a $G$-torsor over $X$ is nothing but a $G(k)$-covering over $X$. Therefore, from the universality of Grothendieck's \'etale fundamental group $\pi_1^{\et}(X,\overline{x})$, there exists a $k$-homomorphism of  $\pi^N(X,x)$ to the pro-constant group scheme associated with $\pi_1^{\et}(X,\overline{x})$. In fact, this homomorphism is surjective. Furthermore, if $k$ is of characteristic zero, then it is an isomorphism. For details, see \cite[Remark 2.10]{ehs08}.
\end{rem}

\subsection{The maximal linearly reductive quotient of $\pi^N$}\label{sec:lin red}

Now let us recall the maximal \textit{linearly reductive} quotient $\pi^{\lr}(X,x)$ of $\pi^N(X,x)$ (cf. \cite{bv15}). 

\begin{definition}(cf.~\cite[Section 2]{aov08}) A finite $k$-group scheme $G$ is said to be \textit{linearly reductive} if one of the following equivalent conditions is satisfied:\\
(a) The functor $\Rep(G)\to\Vect;~V\mapsto V^G$ is exact;\\
(b) The category $\Rep(G)$ is semi-simple.
\end{definition}

\begin{prop}\label{prop:aov}(Abramovich-Olsson-Vistoli~\cite[Proposition 2.13]{aov08})
A finite $k$-group scheme $G$ is linearly reductive if and only if for an algebraic closure $\overline{k}$ of $k$, then $G_{\overline{k}}:=G\times_k\overline{k}$ is isomorphic to a semi-direct product $H\ltimes\Delta$ where $H$ is a finite constant $\overline{k}$-group scheme of order prime to the characteristic of $k$ and $\Delta$ is a finite diagonalizable $\overline{k}$-group scheme.
\end{prop}

Here a finite group scheme $G$ is said to be \textit{diagonalizable} if it is abelian and its Cartier dual is constant~(cf.~\cite[Section 2.2]{wa79}). Proposition \ref{prop:aov} does not require the assumption that $k$ is perfect.

\begin{rem}
Assume that $k$ is an algebraically closed field of characteristic $p>0$. From Proposition \ref{prop:aov}, we can deduce that if a finite \'etale group scheme $G$ is linearly reductive if and only if $p\nmid\#G(k)$ and that a finite local $k$-group scheme is linearly reductive if and only if $G\overset{\simeq}{\to}\Diag(\X(G))$ with $\X(G)=\Hom(G,\G_m)$ a $p$-group.
\end{rem}

\begin{definition}(cf.~\cite[Section 10]{bv15})
We denote by $\pi^{\lr}(X,x)$ the maximal linearly reductive quotient of $\pi^N(X,x)$. 
\end{definition}

\begin{rem}
(1) In \cite{bv15}, Borne-Vistoli studied the linearly reductive quotient in terms of \textit{fundamental gerbes}. They called it the \textit{tame fundamental gerbe}.
Here, the word ``tame'' stems from the notion of \textit{tame stacks}~(cf.~\cite{aov08}).\\
(2) If $k$ is of characteristic zero, then any finite $k$-group scheme is linearly reductive and $\pi^N(X,x)=\pi^{\lr}(X,x)$.
\end{rem}

From now on assume that $k$ is of positive characteristic $p>0$. Then a finite $k$-group scheme $G$ is said to be \textit{local} if it is connected. We denote by $\pi^{\loc}(X,x)$ the maximal local quotient of $\pi^N(X,x)$. 
The arguments in \cite{un10} then imply that $\pi^{\loc}(X,x)$ does not depend on the choice of a rational point $x\in X(k)$. More precisely, let  $L(X)$ be the category of pairs $(P,G)$ where $G$ is a finite local $k$-group scheme and $P\to X$ is a $G$-torsor. Then the projective limit
\begin{equation*}
\varprojlim_{(P,G)\in L(X)}(P,G)
\end{equation*}
exists and for any $x\in X(k)$, there exists a canonical isomorphism
\begin{equation*}
\pi^{\loc}(X,x)\simeq\varprojlim_{(P,G)\in L(X)}G.
\end{equation*} 
In particular, for any finite local $k$-group scheme $G$, the map in Proposition \ref{prop:univ} induces a bijection:
\begin{equation}\label{eq:hom-coh}
\Hom_k(\pi^{\loc}(X,x),G)\xrightarrow{\simeq} H^1_{{\rm fpqc}}(X,G).
\end{equation}
Hence, we write simply $\pi^{\loc}(X)$ instead of $\pi^{\loc}(X,x)$.

If $k$ is algebraically closed, then the maximal linearly reductive quotient $\pi^{\loc}(X)^{\lr}$ is diagonalizable and we have:
\begin{equation*}
\pi^{\loc}(X)^{\lr}=\Diag(\X(\pi^{\loc}(X))).
\end{equation*}
Here $\X(\pi^{\loc}(X))=\Hom(\pi^{\loc}(X),\G_m)$, the group of characters of $\pi^{\loc}(X)$ and for any abelian group $A$, we denote by $\Diag(A)$ the diagonalizable $k$-group scheme associated with $A$~\cite[Section 2.2]{wa79}. On the other hand, any homomorphism $\pi^{\loc}(X)\to\G_m$ factors through $\mu_{p^n}=\Diag(\Z/p^n\Z)$ for some integer $n>0$. Hence,
\begin{equation*}
\X(\pi^{\loc}(X))=\varinjlim_{n>0}\Hom_k(\pi^{\loc}(X),\mu_{p^n})\xrightarrow{\simeq}\varinjlim_{n>0}H^1_{\fpqc}(X,\mu_{p^n}).
\end{equation*}
Here, since all the $\mu_{p^n}$ are abelian, the map in Proposition \ref{prop:univ} induces the last isomorphism.
Hence, we have seen that the following holds:

\begin{prop}\label{prop:max loc lr}
If $k$ is an algebraically closed field of characteristic $p>0$, then there exists a canonical isomorphism:
\begin{equation*}
\pi^{\loc}(X)^{\lr}\simeq\varprojlim_{n>0}\Diag(H_{\fpqc}^1(X,\mu_{p^n})).
\end{equation*}
\end{prop}

Note that $\pi^{\loc}(X)^{\lr}$ is nothing but the maximal local linearly reductive quotient of $\pi^N(X,x)$.

%%%%%%%%%%  Statement   %%%%%%%%%%%%%%

\section{A purely inseparable analogue of the Abhyankar conjecture}\label{sec:a purely insep}

Let $k$ be an algebraically closed field of positive characteristic $p>0$. Let $X$ be a projective smooth curve over $k$ of genus $g\ge 0$. Let $U$ be a nonempty open subset of $X$ with $n\Def\# (X\setminus U)> 0$. The scheme $U$ is then an affine smooth curve over $k$. We denote by $\gamma$ the $p$-\textit{rank} of the Jacobian variety ${\rm Pic}_X^0$ of $X$, i.e, 
\begin{equation*}
\gamma\Def\Dim_{\F_p}{\rm Pic}^0_{X}[p](k).
\end{equation*}
Since $X$ is smooth and projective, the invariant $\gamma$ coincides with the dimension of the $\F_p$-vector space $\Hom_{\Z}(\pi_1^{\et}(X)^{\ab},\F_p)$~(cf.~\cite{bouw00}). Moreover, in this case, for any integer $m>0$, we have
\begin{equation}\label{eq:isom mu_p^m}
H^1_{\fpqc}(X,\mu_{p^m})\simeq{\rm Pic}_X^0[p^m](k)\simeq(\Z/p^m\Z)^{\oplus\gamma}.
\end{equation}
Here, for the first equality, see, for example, \cite[Proposition 3.2]{an11}; for the second equality, see, for example, \cite[Chapter IV]{mu08}.

Let $\pi_A^{\loc}(U)$ be the set of isomorphism classes of  finite local $k$-group schemes $G$ such that there exists a surjective homomorphism $\pi^{\loc}(U)\twoheadrightarrow G$.

\subsection{Question}\label{sec:nece cond}

We first give a necessary condition for a finite local $k$-group scheme $G$ to belong to the set $\pi^{\loc}_A(U)$:

\begin{prop}\label{prop:necessarily cond}
For any finite local $k$-group scheme $G$, if $G\in\pi^{\loc}_A(U)$, then there exists an injective homomorphism $\X(G)\hookrightarrow (\Q_p/\Z_p)^{\oplus \gamma+n-1}$.
\end{prop}

By the virtue of Proposition \ref{prop:max loc lr}, Proposition \ref{prop:necessarily cond} is an immediate consequence of the following:

\begin{prop}
For any integer $m>0$, we have
\begin{equation*}
H^1_{\fpqc}(U,\mu_{p^m})\simeq (\Z/p^m\Z)^{\oplus \gamma+n-1}.
\end{equation*}
\end{prop}

\begin{proof}(cf.~\cite[Tag 03RN, Lemma 53.68.3]{stack})
Let $X\setminus U=\{x_1,\dots,x_n\}$. Then there exists an isomorphism ${\rm Pic}(U)\simeq{\rm Pic}(X)/R$ with $R\Def\langle\calO_X(x_i)
\,|\,1\le i\le n\,\rangle_{\Z}\subset{\rm Pic}(X)$. Therefore,
\begin{equation*}
\begin{aligned}
H^1_{\fpqc}(U,\mu_{p^m})~&\simeq~\{(L,\alpha)\,|\,L\in{\rm Pic}(U),\alpha:L^{\otimes p^m}\xrightarrow{\simeq}\calO_U\}/\simeq\\
~&\simeq~\biggl\{(\overline{L},D,\overline{\alpha})\,\biggl |\,
\begin{aligned}
\overline{L}\in{\rm Pic}(X)&,~D\in\Z x_1+\cdots+\Z x_n,~\\
\overline{\alpha}:&\overline{L}^{\otimes p^m}\xrightarrow{\simeq}\calO_X(D)
\end{aligned}
\biggl\}/\overline{R},
\end{aligned}
\end{equation*}
where $\overline{R}$ is the group defined by 
\begin{equation*}
\overline{R}\Def\{(\calO_X(D'),p^mD',{\rm id})\,|\,D'\in\Z x_1+\cdots+\Z x_n\}.
\end{equation*}
We identify the group $H^1_{\fpqc}(U,\mu_{p^m})$ with the second one in the right hand side of the above equation.
We then obtain the following exact sequence
\begin{equation*}
0\to H^1_{\fpqc}(X,\mu_{p^m})\to H^1_{\fpqc}(U,\mu_{p^m})\to\bigoplus_{i=1}^n\Z/p^m\Z
\xrightarrow{\Sigma}\Z/p^m\Z\to 0.
\end{equation*}
Here the second map is given by $(\overline{L},\alpha)\mapsto (\overline{L},0,\alpha)$ and the third one by $(\overline{L},D,\alpha)\mapsto (\overline{a_i})_{i=1}^n$ with $D=\sum_{i=1}^n a_i x_i$. This completes the proof~(cf.~(\ref{eq:isom mu_p^m})).
\end{proof}

Considering the Abhyankar Conjecture~(cf.~Theorem \ref{thm:abhy}), the following question naturally arises:

\begin{ques}\label{ques:purely insep}
Let $G$ be a finite local $k$-group scheme.
If there exists an injective homomorphism $\X(G)\hookrightarrow (\Q_p/\Z_p)^{\oplus \gamma+n-1}$, then does the group scheme $G$ belong to the set $\pi^{\loc}_A(U)$?
\end{ques}

\subsection{Nilpotent case}

Now we will show that, for any finite local  nilpotent $k$-group scheme $G$, Question \ref{ques:purely insep} has an affirmative answer:

\begin{prop}\label{prop:nilpotent}
Let $G$ be a finite local nilpotent $k$-group scheme. Then $G\in\pi^{\loc}_A(U)$ if and only if there exists an injective homomorphism $\X(G)\hookrightarrow (\Q_p/\Z_p)^{\oplus \gamma+n-1}$.
\end{prop}

\begin{proof}
First we remark on $\alpha_p$-torsors over $U$. Since $U$ is affine,
\begin{equation*}
H^1_{\fpqc}(U,\G_a)=H^1(U,\mathcal{O}_U)=0,
\end{equation*}
we have:
\begin{equation*}
H^1_{\fpqc}(U,\alpha_p)\simeq \Gamma(U,\mathcal{O}_U)/\Gamma(U,\mathcal{O}_U)^p.
\end{equation*}
On the other hand, $U$ is an affine smooth integral scheme, there exists a dominant morphism  $U\to\A_k^1$, whence
\begin{equation}\label{eq:a_p}
H^1_{\fpqc}(U,\alpha_p)\hookleftarrow H^1_{\fpqc}(\mathbb{A}_k^1,\alpha_p)=\oplus_{p \nmid n}k\cdot t^n,
\end{equation}
where $t$ is the coordinate of $\mathbb{A}^1$. Furthermore, since $\alpha_p$ is simple, any nonzero element of $H^{1}_{\fpqc}(U,\alpha_p)$ corresponds to a surjective homomorphism $\pi^{\loc}(U)\twoheadrightarrow \alpha_p$~(cf.~(\ref{eq:hom-coh})).

Now let us prove the proposition. It suffices to show the `if' part. We prove this by induction on the order $\Dim_k k[G]=p^{r}~(r>0)$. From the assumption, $G$ is obtained by central extensions of $\alpha_p$ or $\mu_p$. If $\Dim_k k[G]=p$, then $G=\alpha_p$, or $=\mu_p$ and the statement is immediate from (\ref{eq:a_p}), or from the assumption. Since $G$ is a nontrivial nilpotent group scheme, the center $Z(G)$ is nontrivial. Let $H\subset Z(G)$ be a subgroup scheme of order $p$. Then we get a central extension of finite $k$-group schemes:
\begin{equation}\label{eq:central ext}
1\to H\to G\to G/H\to 1.
\end{equation}
Since $\Dim_k k[G/H]<\Dim_k k[G]$ and  $\X(G/H)\subseteq\X(G)$, by induction hypothesis, there exists a surjective homomorphism $\overline{\phi}:\pi^{\loc}(U)\to G/H$. Since $U$ is affine, $H^q_{\fpqc}(U,\G_a)=0$ if $q\neq 0$, we have $H^2_{\fpqc}(U,\alpha_p)=0$. On the other hand, $H^1_{\fpqc}(U,\G_m)={\rm Pic}(U)$ is divisible (cf.~\cite[Tag 03RN, Proof of Lemma 53.68.3]{stack}) and $H^2_{\fpqc}(U,\G_m)={\rm Br}(U)=0$, we then also have $H^2_{\fpqc}(U,\mu_p)=0$. Therefore, we find that $H^2_{\fpqc}(U,H)=0$ and the exactness  of (\ref{eq:central ext}), noticing that $H_{\fpqc}^0(U,G/H)=0$, implies that the resulting sequence
\begin{equation*}
0\to H^1_{\fpqc}(U,H)\to H^1_{\fpqc}(U,G)\to H^1_{\fpqc}(U,G/H)\to 0
\end{equation*}
is an exact sequence of pointed sets (cf.~\cite[p.~284,~Remarque 4.2.10]{gi71}). Therefore, the isomorphism (\ref{eq:hom-coh}) implies that there exists a lift $\phi:\pi^{\loc}(U)\to G$ of $\overline{\phi}$ and we obtain the following commutative diagram:
\begin{equation*}
\begin{xy}
\xymatrix{
0\ar[r]& K\ar[r]\ar[d]^{f}& \pi^{\loc}(U)\ar[r]^{\overline{\phi}}\ar[d]^{\phi}& G/H\ar[r]\ar @{=} [d]& 0\\
0\ar[r]& H\ar[r]& G\ar[r]& G/H\ar[r] & 0,
}
\end{xy}
\end{equation*}
where $K\Def\Ker(\overline{\phi})$. If $f$ is nontrivial, then it is surjective, whence so is $\phi$. Thus from now on assume $f=0$. In this case, the homomorphism $\phi$ factors through $G/H$. Thus, the central extension (\ref{eq:central ext}) is trivial, i.e., $G= H\times (G/H)$. We claim that
\begin{equation}\label{eq:nilpotent claim}
H^1_{\fpqc}(U,H)\supsetneq \Ker\bigl(\Hom_k(\pi^{\loc}(U),H)
\to\Hom_k (K,H)\bigl).
\end{equation}
If the claim (\ref{eq:nilpotent claim}) is true, then one take a $k$-homomorphism $g:\pi^{\loc}(U)\to H$ so that $g|_K\neq 0$ and the one $(g,\overline{\phi}):\pi^{\loc}(U)\to G=H\times (G/H)$ is surjective. Thus it remains to show the claim (\ref{eq:nilpotent claim}). Notice that
\begin{equation*}
\Hom_k(G/H,H)=\Ker\bigl(\Hom_k(\pi^{\loc}(U),H)
\to\Hom_k (K,H)\bigl).
\end{equation*}
Thus, if $H\simeq\alpha_p$, then the claim (\ref{eq:nilpotent claim}) follows from $\Dim_k H^1_{\fpqc}(U,\alpha_p)=\infty$. If $H\simeq \mu_p$, then the claim (\ref{eq:nilpotent claim}) follows from the following inequality:
\begin{equation*}
\Dim_{\F_p}\Hom_k(G/H,\mu_p)<\Dim_{\F_p}\Hom_k (G,\mu_p)\le\gamma+n-1=\Dim_{\F_p}H^1_{\fpqc}(U,\mu_p).
\end{equation*}
Here for the first inequality, we use $G\simeq\mu_p\times (G/H)$. This completes the proof.
\end{proof}

\begin{cor}
Every finite local unipotent $k$-group scheme appears as a finite quotient of $\pi^{\loc}(U)$.
\end{cor}

\begin{ex}\label{ex:SL_2}
Assume $k$ is of characteristic $p=2$. In this case, the first Frobenius kernel
\begin{equation*}
\SL_{2(1)}\Def\Ker\bigl(F^{(1)}:\SL_2^{(-1)}\to \SL_2\bigl)
\end{equation*}
of the algebraic group $\SL_2$ is nilpotent. Indeed, noticing that
\begin{equation*}
\SL_{2(1)}(A)=\biggl\{
\begin{pmatrix}
a & b\\
c & d
\end{pmatrix}
\biggl |\,
\begin{aligned}
a,b,c,d\in A,\,
ad-bc=1\,\\
a^2=d^2=1,\,b^2=c^2=0\,
\end{aligned}
\biggl\}
\end{equation*}
for any $k$-algebra $A$, the maps
\begin{equation*}
\SL_{2(1)}(A)\to (\alpha_2\times\alpha_2)(A);~ \begin{pmatrix}
a & b\\
c & d
\end{pmatrix}
\mapsto (ab,cd),\quad A:\text{$k$-algebra}
\end{equation*}
then form a $k$-homomorphism $\SL_{2(1)}\to \alpha_2\times\alpha_2$, which makes the following sequence
\begin{equation}\label{eq:SL_2}
1\to\mu_2\to\SL_{2(1)}\to\alpha_2\times\alpha_2\to 1
\end{equation}
a nonsplit central extension~(cf.~\cite[Chapter 10, Exercise 3]{wa79}). In particular, $\SL_{2(1)}$ is nilpotent. From the facts  that $\X(\mu_2)=\Z/2\Z$ and that any homomorphism $\SL_{2(1)}\to\G_m$ factors through $\mu_2$, the nonsplitness of (\ref{eq:SL_2}) deduces the condition that $\X(\SL_{2(1)})=\X(\alpha_2\times\alpha_2)=1$. Therefore,  by applying Proposition \ref{prop:nilpotent}, we can conculde that there exists a surjective homomorphism $\pi^{\loc}(\mathbb{A}^1_k)\twoheadrightarrow\SL_{2(1)}$.
\end{ex}

\begin{rem}\label{rem:purely insep}
(1) In the particular case where $U=\A^1_k$, that Question \ref{ques:purely insep} is affirmative is equivalent to the assertion that any finite local $k$-group scheme $G$ with $\X(G)=1$ appears as a quotient of $\pi^{\loc}(\A^1_k)$. For example, for any integers $n,r>0$, the $r$-th Frobenius kernel $\SL_{n(r)}\Def\Ker\bigl(F^{(r)}:\SL_n^{(-r)}\to\SL_n\bigl)$ of $\SL_n$ gives such one, i.e., $\X(\SL_{n(r)})=1$~(cf.~(3) below).\\
(2) In general, if a finite $k$-group scheme $G$ is generated by all the unipotent subgroup schemes, then $G$ has no characters, i.e., $\X(G)=1$. The author expects that the converse might be true, namely, that $\X(G)=1$ if and only if $G$ is generated by all the unipotent subgroup schemes.\\
(3) Let us see another example of finite local $k$-group scheme $G$ with $\X(G)=1$. Let $\Sigma$ be a semi-simple simply connected algebraic group over $k$. Then for any integer $r>0$, the $r$-th Frobenius kernel
\begin{equation*}
\Sigma_{(r)}\Def\Ker(\Sigma^{(-r)}\xrightarrow{F^{(r)}}\Sigma)
\end{equation*}
has no nontrivial characters, i.e., $\X(\Sigma_{(r)})=1$~(Indeed, since $\Sigma$ is semi-simple simply connected, any character $\Sigma_{(r)}\to\G_m$ comes from some character of $\Sigma$~\cite[Part II, Chapter 3,~3.15, Proposition and Remarks~2)]{ja03}. However, since $\Sigma$ is semi-simple, there exist no nontrivial characters of $\Sigma$~\cite[Part II, Chapter 1, 1.18~(3)]{ja03}). Therefore, if Question \ref{ques:purely insep} is affirmative for the affine line $\A^1_k$, then the group scheme $\Sigma_{(r)}$ must appear as a finite quotient of $\pi^{\loc}(\A_k^1)$. We will prove this fact is actually true~(cf.~Corollary \ref{cor:bertini}).\\
(4) Moreover, if $\Sigma$ is semi-simple simply connected algebraic group over $k$, then one can prove that the 1-st Frobenius kernel $\Sigma_{(1)}$ is generated by all the unipotent subgroup schemes of it. Indeed, fix a maximal torus $T<\Sigma$. Let $R$ be the root system. Choose a positive system $R^+\subset R$ and denote by $S$ the corresponding set of simple roots~(\cite[Part II, Chapter 1, 1.5]{ja03}). We denote by $U^{+}$ (resp. $U^-$) the unipotent radical corresponding to the positive roots (resp. the negative roots). By \cite[Part II, Chapter 3, 3.2~Lemma]{ja03}, it suffices to show that $T_{(1)}\subset\langle U^{\pm}_{(1)}\rangle$. Since $\Sigma$ is simply connected, we have
\begin{equation*}
(\alpha^{\vee})_{\alpha\in S}:\prod_{\alpha\in S}\G_m\xrightarrow{\simeq} T
\end{equation*}
(cf.~\cite[Part II, Chapter 1, 1.6~(4)]{ja03}), where the $\alpha^{\vee}$ are dual roots. Hence, we are reduced to show that $\alpha^{\vee}(\mu_{p})\subset\langle U^{\pm}_{(1)}\rangle$ for any $\alpha\in S$. For this, we may assume that $\Sigma=\SL_2$. In this case, we have $\Dim~k[\SL_{2(1)}]=p^3$ and $\Dim~k[U^{\pm}_{(1)}]=p$, whence  $p^2\le\Dim~k[\langle U^{\pm}_{(1)}\rangle]\le p^3$. For the equality $\SL_{2(1)}=\langle U^{\pm}_{(1)}\rangle$, it suffices to show $\Dim~k[\langle U^{\pm}_{(1)}\rangle]=p^3$. However, again by \cite[Part II, Chapter 3, 3.2~Lemma]{ja03}, there exists a  surjective $k$-algebra homomorphism $k[\SL_{2(1)}]\twoheadrightarrow k[U^+_{(1)}\times U^-_{(1)}]$. This factors through $k[\langle U^{\pm}_{(1)}\rangle]$ and the resulting algebra map $\phi:k[\langle U^{\pm}_{(1)}\rangle]\to k[U^+_{(1)}\times U^-_{(1)}]$ is then surjective. The map $\phi$ is not isomorphism because $U^+_{(1)}\times U^-_{(1)}$ is not a subgroup scheme of $\SL_{2(1)}$ but $\langle U^{\pm}_{(1)}\rangle$ is. Therefore, we have $\Dim~k[\langle U^{\pm}_{(1)}\rangle]>\Dim~k[U^+_{(1)}\times U^-_{(1)}]=p^2$. Then we must have $\Dim~k[\langle U^{\pm}_{(1)}\rangle]=p^3$, which implies that $\SL_{2(1)}=\langle U^{\pm}_{(1)}\rangle$. This completes the proof.
\end{rem}

%%%%%%%%%%  Main result   %%%%%%%%%%%%%%

\section{Main results}\label{sec:main}

\subsection{Torsors coming from Frobenius endmorphisms of an affine algebraic group}\label{sec:main-prob}

Let $k$ be an algebraically closed field of characteristic $p>0$ and $U$ a smooth affine curve over $k$. Let $\Sigma$ be an affine algebraic group over $k$. Note that for each integer $r>0$, the $r$-th relative Frobenius morphism $F^{(r)}:\Sigma^{(-r)}\to\Sigma$ gives a saturated $\Sigma_{(r)}\Def\Ker(F^{(r)})$-torsor over $\Sigma$. Here, recall that the saturatedness means the corresponding homomorphism $\pi^N(\Sigma)\to\Sigma_{(r)}$ is surjective~(cf.~Definition \ref{definition:saturated}).
In Section \ref{sec:main}, motivated by Question \ref{ques:purely insep}, we will consider the following question:

\begin{ques}\label{ques:main}
Fix an integer $r>0$.\\
(1) Does there exist any $k$-morphism $f:U\to\Sigma$ so that the $\Sigma_{(r)}$-torsor $f^*\Sigma^{(-r)}\to U$ defined by the following cartesian diagram is saturated~? 
\begin{equation*}
\begin{xy}
\xymatrix{\ar@{}[rd]|{\square}
f^*\Sigma^{(-r)}\ar[r]\ar[d]&\Sigma^{(-r)}\ar[d]^{F^{(r)}}\\
U\ar[r]_{f}&\Sigma.
}
\end{xy}
\end{equation*}
(2) Furthermore, if exists, for which $k$-morphism $f:U\to\Sigma$, is the resulting $\Sigma_{(r)}$-torsor $f^*\Sigma^{(-r)}\to U$ saturated~?
\end{ques}

Let us begin with showing that one can reduce the problem to the case $r=1$:

\begin{lem}\label{lem:any ht}
Let $\Sigma$ be an affine algebraic group over $k$ and $f:U\to\Sigma$ a $k$-morphism. If $f^*\Sigma^{(-1)}$ is saturated, then for any integer $r>1$, $f^*\Sigma^{(-r)}$ is also saturated. 
\end{lem}

\begin{proof}
We will show this by induction on $r\ge 1$. We will denote by $\phi^{(-r)}:\pi^{\loc}(U)\to\Sigma_{(r)}$ the homomorphism corresponding to the torsor $f^*\Sigma^{(-r)}$. Assume $\phi^{(-r)}$ is surjective. Let us show that $\phi^{(-r-1)}$ is also surjective. Since $F^{(r)*}f^*\Sigma^{(-r)}$ is a trivial torsor,   the composition
\begin{equation*}
\pi^{\loc}(U^{(-r)})\xrightarrow{F^{(r)}_*}\pi^{\loc}(U)\overset{\phi^{(-r)}}{\twoheadrightarrow}\Sigma_{(r)}
\end{equation*}
is trivial. We then obtain the following commutative diagram:
\begin{equation}\label{eq:any ht}
\begin{xy}
\xymatrix{
&\pi^{\loc}(U^{(-r)})\ar[r]^{F^{(r)}_*}\ar @{-->} [d]_{\psi~ :=}^{\exists}&\pi^{\loc}(U)\ar[d]_{\phi^{(-r-1)}}\ar[rd]^{\phi^{(-r)}} & & \\
1\ar[r]& \Sigma_{(1)}^{(-r)}\ar[r]& \Sigma_{(r+1)}\ar[r]&\Sigma_{(r)}\ar[r]& 1~(\text{exact})
}
\end{xy}
\end{equation}
and the map $\phi^{(-r-1)}\circ F^{(r)}_*$ factors through $\Sigma_{(1)}^{(-r)}$. We denote by $\psi$ the resulting homomorphism $\pi^{\loc}(U^{(-r)})\to\Sigma_{(1)}^{(-r)}$. We are then reduced to showing the surjectivity of $\psi$. The commutativity  of the diagram (\ref{eq:any ht}) implies that ${\rm Ind}_{\Sigma_{(1)}^{(-r)}}^{\Sigma_{(r+1)}}(Q)\simeq F^{(r)*}f^*\Sigma^{(-r-1)}$, where $Q$ is the torsor over $U^{(-r)}$ corresponding to the morphism $\psi$. On the other hand, by considering the tautological commutative diagram
\begin{equation*}
\begin{xy}
\xymatrix{
\Sigma^{(-r-1)}\ar[r]^{=}\ar[d]_{F^{(1)}}&\Sigma^{(-r-1)}\ar[d]^{F^{(r+1)}}\\
\Sigma^{(-r)}\ar[r]^{F^{(r)}}&\Sigma,
}
\end{xy}
\end{equation*}
we can find that
\begin{equation*}
{\rm Ind}_{\Sigma_{(1)}^{(-r)}}^{\Sigma_{(r+1)}}(\Sigma^{(-r-1)}\xrightarrow{F^{(1)}}\Sigma^{(-r)})=F^{(r)*}(\Sigma^{(-r-1)}\xrightarrow{F^{(r+1)}}\Sigma).
\end{equation*}
Therefore, by the construction, $Q$ is nothing but the torsor defined by the cartesian diagram
\begin{equation*}
\begin{xy}
\xymatrix{\ar@{}[rd]|{\square}
Q\ar[r]\ar[d]&\Sigma^{(-r-1)}\ar[d]^{F^{(1)}}\\
U^{(-r)}\ar[r]^{f^{(-r)}}&\Sigma^{(-r)},
}
\end{xy}
\end{equation*}
where $f^{(-r)}$ is just the $r$-th Frobenius twist of $f$:
\begin{equation*}
\begin{xy}
\xymatrix{\ar@{}[rd]|{\square}
U^{(-r)}\ar[r]^{\simeq}\ar[d]_{f^{(-r)}}&U\ar[d]^{f}\\
\Sigma^{(-r)}\ar[r]^{\simeq}&\Sigma.
}
\end{xy}
\end{equation*}
Then, the saturatedness of the torsor $f^*\Sigma^{(1)}\to U$ indicates the saturatedness of the torsor $Q\to U^{(-r)}$, or equivalently, the surjectivity of $\psi$. This completes the proof.
\end{proof}

Next let us see a basic example. 
The following proposition gives a complete answer to Question \ref{ques:main} for the pairs $(U,\Sigma)=(\A_k^1,\G_a^{\oplus n})~(n\ge 1)$.

\begin{prop}\label{prop:toy example}
Assume that $k$ is of characteristic $p>0$ and $n>0$ an integer. Let
\begin{equation*}
\underline{f}=(f_i(t))\in k[t]^{\oplus n}={\rm Mor}_k\bigl(\A^1_k,\G_a^{\oplus n}\bigl)
\end{equation*}
be a $k$-morphism $\A^1_k\to\G_a^{\oplus n}$. We define the  $\alpha_p^{\oplus n}$-torsor $P_{\underline{f}}$ over $\A^1_k$ by the pulling back of the relative Frobenius morphism $F^{(1)}:{\G_a^{(-1)}}^{\oplus n}\to\G_a^{\oplus n}$, i.e., $P_{\underline{f}}\Def \underline{f}^*{\G_a^{(-1)}}^{\oplus n}$. Then $P_{\underline{f}}$ is saturated if and only if the images $\overline{f_i(t)}\,(\,1\le i\le n)$ in $H^1_{\fpqc}(\A^1_k,\alpha_p)=k[t]/k[t^p]$ are linearly independent over $k$.
\end{prop}

\begin{proof}
We will show this by induction on $n>0$. In the case where $n=1$, then since $\alpha_p$ is simple, the assertion is obvious. 
From now on assume that $n>1$ and that $ \Dim_k\langle\overline{f_i(t)}\,|\,1\le i\le n-1\rangle_k=n-1$. 
Put $\underline{f'}:=(f_1(t),\dots,f_{n-1}(t))$. Then $P_{\underline{f'}}$ is an $\alpha_p^{\oplus n-1}$-torsor over $\A_k^1$. We denote by $\phi$ and $\phi'$ the homomorphism $\pi^{\loc}(\A^1_k)\to\alpha_p^{\oplus n}$ corresponding to $P_{\underline{f}}$ and the one $\pi^{\loc}(\A^1_k)\to\alpha_p^{\oplus n-1}$ to $P_{\underline{f'}}$, respectively. From the assumption, $\phi'$ is surjection. Let $K\Def\Ker(\phi')$.  We then obtain the following commutative diagram:
\begin{equation*}
\begin{xy}
\xymatrix{
0\ar[r]& K\ar[r]\ar[d]^{\psi}& \pi^{\loc}(\A^1_k)\ar[r]^{\phi'}\ar[d]^{\phi}& \alpha_p^{\oplus n-1}\ar[r]\ar @{=} [d]& 0\\
0\ar[r]&\alpha_p\ar[r]& \alpha_p^{\oplus n}\ar @/^5mm/ [l]^{{\rm pr_n}}\ar[r]& \alpha_p^{\oplus n-1}\ar[r] & 0.
}
\end{xy}
\end{equation*}
Then we have
\begin{equation*}
\begin{aligned}
\phi:~\text{surjective}~&\Longleftrightarrow~\psi~\neq~1\\
~&\Longleftrightarrow~{\rm pr_n}\circ\phi\in\Hom(\pi^{\loc}(\A^1_k),\alpha_p)\setminus{\phi'}^*\Hom(\alpha_p^{\oplus n-1},\alpha_p)\\
~&\Longleftrightarrow~f_n(t)\in (k[t]/k[t^p])\setminus\langle\overline{f_1(t)},\dots,\overline{f_{n-1}(t)}\rangle_k\\
~&\Longleftrightarrow~\Dim_k\langle\overline{f_i(t)}\,|\,1\le i\le n\rangle_k=n.
\end{aligned}
\end{equation*}
This completes the proof.
\end{proof}

\subsection{Explicit equations definig saturated $\SL_{2(1)}$-torsors in characteristic $p=2$ case}

We will continue to use the same notation as in Section \ref{sec:main-prob}. As we have seen in Example \ref{ex:SL_2}, in the case where $k$ is of characteristic $p=2$, there exists a saturated $\SL_{2(1)}$-torsor $P\to\A^1_k$~(cf. Definition \ref{definition:saturated}). On the other hand, since $H^1_{\fpqc}(\A^1_k,\GL_2)=0$ and $H^0_{\fpqc}(\A^1_k,\GL_2)\to H^0_{\fpqc}(\A^1_k,\G_m)=k^*$ is surjective, we have $H^1_{\fpqc}(\A_k^1,\SL_2)=0$, whence
\begin{equation*}
H^1_{\fpqc}(\A^1_k,\SL_{2(1)})\simeq \SL_2(k[t^2])\backslash\SL_2(k[t]).
\end{equation*}
Therefore, such a torsor $P\to\A^1_k$ must be obtained by the pulling back of the relative Frobenius morphism $F^{(1)}:\SL_2^{(-1)}\to\SL_2$ along some $k$-morphism $f:\A^1_{k}\to\SL_2$:
\begin{equation*}
\begin{xy}
\xymatrix{\ar@{}[rd]|{\square}
P\ar[r]\ar[d]&\SL_2^{(-1)}\ar[d]^{F^{(1)}}\\
\A^1_k\ar[r]^{f}&\SL_2.
}
\end{xy}
\end{equation*}

Hence, by Proposition \ref{prop:nilpotent} and Example \ref{ex:SL_2}, combining with Lemma \ref{lem:any ht}, we can obtain an affimative answer to Question \ref{ques:main}(1) for the pair $(U,\Sigma)=(\A_k^1,\SL_2)$ in the case where $p=2$. As a consequence of it, we have:

\begin{cor}
Assume $p=2$. Then there exists a surjective homomorphism 
\begin{equation*}
\pi^{\loc}(\A^1_k)\twoheadrightarrow\varprojlim_{r>0}\SL_{2(r)}.
\end{equation*}
In particular, for any integer $r>0$, there exists a surjective homomorphism $\pi^{\loc}(\A_k^1)\twoheadrightarrow\SL_{2(r)}$.
\end{cor}

Next, we will consider Question \ref{ques:main}(2) to the pair $(U,\Sigma)=(\A_k^1,\SL_{2})$ in the case where char $k=2$ and will give an answer~(cf.~Corollary \ref{cor:SL_2}).

Recall that the saturatedness of a finite \'etale torsor $P\to U$ depends only on the underlying scheme of it. In fact, it is saturated if and only if it is (geometrically) connected~(cf.~\cite[Lemma 2.3]{zh13}).
One of difficulties of our problem is that the saturatedness of a local torsor, in contrary to \'etale case, depends also on the multiplicative structure of the underlying group scheme. The following simple example indicates such a situation:

\begin{ex}
Let $l$ be a prime number with $l\neq p$. We define the $k$-morphism $f:\G_m\to\G_m\times\G_m$ by $a\mapsto (a,a^l)$. We define the one $g:\G_m\to\G_a\times\G_m$ as the composition $\G_m\xrightarrow{f}
\G_m\times\G_m\subset\G_a\times\G_m$. Then the underlying schemes of the torsors $f^*(\G_m^{(-1)}\times\G_m^{(-1)})$ and $g^*(\G_a^{(-1)}\times\G_m^{(-1)})$ are isomorphic to each other. However, the former is not saturated, but the latter is.
\end{ex}

Hence, it seems to be difficult to obtain a concise characterization of the saturatedness of finite local torsors purely in terms of the category $N(X,x)$~(cf. Section \ref{sec:preliminaries}). 
To avoid this problem, we will rely on a Tannakian interpretation of $\pi^N(U)$. We will use the category of \textit{generalized stratified bundles}, introduced by Esnault-Hogadi~\cite{eh12}. For the full definition of it, see \cite{eh12} (for the one of stratified bundles in the usual sense, see,  for example, \cite{gi75}\cite{ds07}). By the virtue of Lemma \ref{lem:any ht}, the category $\Str(U,1)$ of \textit{$1$-stratified bundles} is large enough for our purpose:

\begin{definition}
Let $X$ be a smooth algebraic variety over a perfect field $k$ of characteristic $p>0$. A \textit{$1$-stratified bundle} on $X$ is a sequence $\{E^{(i)}\}_{i=0}^{\infty}$ of coherent sheaves $E^{(i)}$ over $X^{(i)}$ together with isomorphisms $\sigma^{(i)}:E^{(i)}\xrightarrow{\simeq} {F^{(1)*}}E^{(i+1)}$ for $i\ge 1$ and
\begin{equation}\label{eq:sigma^0}
\sigma^{(0)}:F^{(1)*}E^{(0)}\xrightarrow{\simeq}
F^{(2)*}E^{(1)}.
\end{equation}
Let $E=\{E^{(i)},\sigma^{(i)}\}, E'=\{{E'}^{(i)},{\sigma'}^{(i)}\}$ be arbitrary two $1$-stratified bundles. 
A \textit{homomorphism} of $E$ into $E'$ is a sequence of $\calO_{X^{(i)}}$-linear homomorphisms $\phi^{(i)}:E^{(i)}\to {E'}^{(i)}~(i\ge 0)$ satisfying
\begin{equation*}
\begin{aligned}
&{\sigma'}^{(i)}\circ \phi^{(i)}~=~F^{(1)*}\phi^{(i+1)}\circ\sigma^{(i)},\quad i\ge 1;\\
&{\sigma'}^{(0)}\circ F^{(1)*}\phi^{(0)}~=~F^{(2)*}\phi^{(1)}\circ\sigma^{(0)}.
\end{aligned}
\end{equation*}
The homomorphisms of $1$-stratified bundles satisfy the composition rule and one obtain the category $\Str(X,1)$ of $1$-stratified bundles on $X$.
\end{definition}

\begin{thm}\label{thm:esnault-hogadi}(Esnault-Hogadi~\cite{eh12})
The category $\Str(X,1)$ of $1$-stratified bundles is a $k$-linear abelian rigid tensor category, and if one take a $k$-rational point $x\in X(k)$, then the functor $\omega_x:\Str(X,1)\to\Vect;~\{E^{(i)},\sigma^{(i)}\}\mapsto x^*E^{(0)}$ defines a neutral fiber functor. Furthermore, the maximal profinite quotient of its Tannakian fundamental group $\pi_1(\Str(X,1),\omega_x)$ coincides with the image of $F^{(1)}:\pi^N(X,x)\to\pi^N(X,x)^{(1)}$.
\end{thm}

In particular, if $G$ is a finite local $k$-group scheme of height one, i.e., $G_{(1)}=G$, then any homomorphism $\phi:\pi^{\loc}(X)\to G$ factors through the maximal pro-finite local quotient $\pi_1(\Str(X,1),\omega_x)^{{\rm prof.}\loc}$. Denote by $\psi$ the resulting homomorphism
\begin{equation*}
\psi:\pi_1(\Str(X,1),\omega_x)^{{\rm prof.}\loc}\to G.
\end{equation*}
Let us consider the composition
\begin{equation}\label{eq:def of h}
h_{\phi}:\Rep(G)\xrightarrow{\psi^*}\Rep(\pi_1(\Str(X,1),\omega_x)^{{\rm prof.}\loc})\subset\Str(X,1).
\end{equation}
Then, from the standard Tannakian argument (cf. \cite[Chapter II, Proposition 3]{no82}), we have:

\begin{lem}\label{lem:criterion}
The homomorphism $\phi$ is surjective if and only if 
\begin{equation*}
\Dim_k\Hom_{\Str(X,1)}(\unit,h_{\phi}(k[G],\rho_{{\rm reg}}))=1.
\end{equation*}
Here $\unit\Def\{\calO_{X^{(i)}},{\rm id}\}$ is a unit object of $\Str(X,1)$.
\end{lem}

\begin{rem}
Since $h_{\phi}(k[G],\rho_{\rm reg})$ is an algebra object in $\Str(U,1)$, the $1$-stratified bundle $h_{\phi}(k[G],\rho_{\rm reg})$ admits a morphism from the unit object $\unit\to h_{\phi}(k[G],\rho_{\rm reg})$ corresponding to the unit element of the algebra. Thus the dimension $\Dim_k\Hom_{\Str(X,1)}(\unit,h_{\phi}(k[G],\rho_{{\rm reg}}))$ is always greater than or equal to $1$.
\end{rem}

If $\pi:P\to X$ is the $G$-torsor corresponding to the homomorphism $\phi:\pi^{\loc}(X)\to G$. Then one can describe the $1$-stratified bundle $h_{\phi}(k[G],\rho_{\rm reg})$ as follows~(cf.~\cite[Construction 4.1]{eh12}). We define $\{E^{(i)}\}_{i=1}^{\infty}$ by
\begin{equation*}
\begin{aligned}
E^{(i)}&\Def\calO_{X^{(i)}}\otimes k[G]=\calO_{X^{(i)}}^{\oplus\Dim k[G]},\quad i\ge 1;\\
E^{(0)}&\Def(\pi_*\calO_P\otimes k[G])^G.
\end{aligned}
\end{equation*}
Here $G$ acts on $\pi_*\calO_P\otimes k[G]$ by
\begin{equation*}
(P\times G)\times G\ni ((p, h),g)\mapsto (p\cdot g^{-1},h\cdot g)\in P\times G.
\end{equation*}
We define $\{\sigma^{(i)}\}_{i=0}^{\infty}$ by $\sigma^{(i)}\Def{\rm id}$ for $i\ge 1$ and $\sigma^{(0)}$ the canonical trivialization morphism:
\begin{equation}\label{eq:trivialization}
F^{(1)*}\bigl((\pi_*\calO_P\otimes k[G])^G\bigl)\xrightarrow{\simeq}
F^{(2)*}E^{(1)}=\calO_{X^{(-1)}}\otimes k[G].
\end{equation}
Here, since $G$ is of height one, the torsor $F^{(1)*}P$ is trivial and admits a canonical section $F^{(1)*}P\supset (F^{(1)*}P)_{\rm red}=X^{(-1)}$. The isomorphism $\sigma^{(0)}$ is the one corresponding to this section. 
We then have $h_{\phi}(k[G],\rho_{\rm reg})=\{E^{(i)},\sigma^{(i)}\}$. Note that $(k[G],\rho_{\rm reg})$ is a $G$-torsor object in $\Rep(G)$ and the functor $h_{\phi}$ is a tensor functor, hence $h_{\phi}(k[G],\rho_{\rm reg})$ gives a $G$-torsor object in $\Str(X,1)$. We also denote $P^{\text{$1$-strat}}\Def h_{\phi}(k[G],\rho_{\rm reg})$. Lemma \ref{lem:criterion} implies that $P$ is saturated if and only if $P^{\text{$1$-strat}}$ is connected as an algebra object in $\Str(X,1)$.

\begin{ex}
Let $m>0$ be an integer and $f:\G_m\to\G_m$ a $k$-morphism defined by $a\mapsto a^m$. Let $P^{\text{$1$-strat}}=\{E^{(i)},\sigma^{(i)}\}$ with $P\Def f^*\G_m^{(-1)}$ a $\mu_p=\Spec k[z]/(z^p-1)$-torsor. Then 
\begin{equation*}
E^{(0)}=\biggl(\frac{k[x,x^{-1},y]}{(y^p-x^{m})}\otimes \frac{k[z]}{(z^p-1)}\biggl)^{\mu_p}=\bigoplus_{j=0}^{p-1}k[x,x^{-1}]\cdot (\overline{yz})^j.
\end{equation*}
Therefore, the representation matrix $A\in \GL_p(k[x^{(-1)\pm 1}])$ of $\sigma^{(0)}$ with respect to the basis $\{\overline{yz}^j\,|\,0\le j\le p-1\}$ and $\{\overline{z}^{j}\,|\,0\le j\le p-1\}$ is given by the diagonal matrix:
\begin{equation*}
A=
\begin{pmatrix}
1 & 0                         &\cdots & 0\\
0 & {x^{(-1)}}^{m} & \cdots &0\\
&\cdots &\cdots &\\
0 & 0 & \cdots & {x^{(-1)}}^{m(p-1)}
\end{pmatrix}.
\end{equation*}
A direct computation then recover the following well-understood result:
\begin{equation*}
\Hom_{\Str(\G_m,1)}(\unit,P^{\text{$1$-strat}})~\simeq~
\begin{cases}
k,  &  \text{if $p\nmid m$},\\
k[\mu_p],  &  \text{if $p\mid m$}.
\end{cases}
\end{equation*}
\end{ex}

\begin{ex}
Let $n>0$ be an integer and $\underline{f}=(f_i(x))\in k[x]^{\oplus n}={\rm Mor}_k(\A^1_k,\G_a^{\oplus n})$ a $k$-morphism~(cf.~Proposition \ref{prop:toy example}). Let $P^{\str}=\{E^{(i)},\sigma^{(i)}\}$ with
\begin{equation*}
P\Def \underline{f}^*\G_a^{(-1)\oplus n}=\Spec \frac{k[x,y_i~(1\le i\le n)]}{(y_i^p-f_i(x)~(1\le i\le n))}
\end{equation*}
an $\alpha_p^{\oplus n}=\Spec k[z_i (1\le i\le n)]/(z_i^p (1\le i\le n))$-torsor. Then
\begin{equation*}
E^{(0)}=\bigoplus_{0\le j_1,\dots,j_n\le p-1}k[x]\cdot (\overline{y_1+z_1})^{j_1}\cdots (\overline{y_n+z_n})^{j_n}.
\end{equation*}
The representation matrix $A\in\GL_{p^n}(k[x^{(-1)}])$ of $\sigma^{(0)}$ with respect to the basis
\begin{equation*}
\{(\overline{y_1+z_1})^{j_1}\cdots (\overline{y_n+z_n})^{j_n}\},\quad
\{\overline{z_1}^{j_1}\cdots\overline{z_n}^{j_n}\}
\end{equation*}
is unipotent and, by a direct computation, one can verify that the condition that  $\{\overline{f_i(x)}\}_{i=1}^n\subset k[x]/k[x^p]$ is linearly independent over $k$ is equivalent to the condition that 
\begin{equation*}
\Dim_k\Hom_{\Str(\A^1_k,1)}(\unit,P^{\str})=1.
\end{equation*}
This then gives another proof of Proposition \ref{prop:toy example}.
\end{ex}

Now let us prove the main result of this subsection :

\begin{thm}\label{thm:main}
Assume $p=2$. Let
\begin{equation*}
\underline{f}=
\begin{pmatrix}
f_{22}(x)x^{-m}&-f_{12}(x)x^{-m}\\
-f_{21}(x)x^{-m}&f_{11}(x)x^{-m}
\end{pmatrix}
\in\GL_2(k[x^{\pm 1}])={\rm Mor}_k(\G_m,\GL_2)
\end{equation*}
be a $k$-morphism with $f_{ij}(x)\in k[x]~(1\le i,j\le 2)$ and 
\begin{equation*}
\det
\begin{pmatrix}
f_{11}(x)&f_{12}(x)\\
f_{21}(x)&f_{22}(x)
\end{pmatrix}
=x^m
\end{equation*}
for some $m\in\Z_{\ge 0}$. Let $P\Def\underline{f}^*\GL_2^{(-1)}$ be the resulting torsor over $\G_m$.\\
(1) In the case $2\mid m$, assume that one of the following conditions is satisfied:
\begin{itemize}
\item $\Dim_k\langle\overline{f_{11}},\,\overline{f_{21}}\rangle=\Dim_k\langle\overline{f_{11}f_{21}},\,\overline{f_{12}f_{22}}\rangle=2$;
\item 
$\Dim_k\langle\overline{f_{12}},\,\overline{f_{22}}\rangle=\Dim_k\langle\overline{f_{11}f_{21}},\,\overline{f_{12}f_{22}}\rangle=2$.
\end{itemize}
Then we have:
\begin{equation*}
\Hom_{\Str(\G_m,1)}(\unit,P^{\str})\simeq k[\mu_2].
\end{equation*}
(2) In the case $2\nmid m$, assume that one of the following
conditions is satisfied:
\begin{itemize}
\item $\Dim_{k}\langle\overline{f_{11}},\,\overline{f_{21}}\rangle=2$,  $\Dim_k\langle\overline{f_{11}f_{21}},\,\overline{f_{12}f_{22}}\rangle=\Dim_k\langle\overline{f_{11}f_{12}},\,\overline{f_{21}f_{22}},\,\overline{x^m}\rangle=3$;
\item $\Dim_{k}\langle\overline{f_{12}},\,\overline{f_{22}}\rangle=2$,  $\Dim_k\langle\overline{f_{11}f_{21}},\,\overline{f_{12}f_{22}}\rangle=\Dim_k\langle\overline{f_{11}f_{12}},\,\overline{f_{21}f_{22}},\,\overline{x^m}\rangle=3$
\end{itemize}
Then we have:
\begin{equation*}
\Hom_{\Str(\G_m,1)}(\unit,P^{\str})=k.
\end{equation*}
\end{thm}

Here, for each $f\in k[x]$, $\overline{f}$ denotes by the image of $f$ in $k[x]/k[x^p]$.

\begin{proof}
Let $P^{\str}=\{E^{(i)},\sigma^{(i)}\}$.
Let $\pi:P\to\G_m$ be the structure morphism. We then have
\begin{equation*}
\pi_*\calO_P=\frac{k[x^{\pm 1},y_{ij}~(i,j=1,2)]}{(
y_{11}^2-f_{22}(x)x^{-m},y_{12}^2-f_{12}(x)x^{-m},
y_{21}^2-f_{21}(x)x^{-m},y_{22}^2-f_{11}(x)x^{-m}
)}.
\end{equation*}
Let $\rho_P$ (resp. $\rho_0$) be the coaction $\pi_*\calO_P\to\pi_*\calO_P\otimes k[\GL_{2(1)}]$ induced by the action $P\times\GL_{2(1)}\to P$ (resp. the one by the trivial action). Let $\iota$ the antipode of $k[\GL_{2(1)}]$. Then the composition 
\begin{equation}\label{eq:G-equiv}
\begin{aligned}
\pi_*\calO_P\otimes k[\GL_{2(1)}]&\xrightarrow{\rho\otimes{\rm id}}
\pi_*\calO_P\otimes k[\GL_{2(1)}]\otimes k[\GL_{2(1)}]\\
&\xrightarrow{{\rm id}\otimes\iota\otimes{\rm id}}
\pi_*\calO_P\otimes k[\GL_{2(1)}]\otimes k[\GL_{2(1)}]
\xrightarrow{{\rm id}\otimes m}\pi_*\calO_P\otimes k[\GL_{2(1)}]
\end{aligned}
\end{equation}
gives a $k[\GL_{2(1)}]$-comodule isomorphism
\begin{equation}\label{eq:G-equiv 2}
(\pi_*\calO_P,\rho_P)\otimes (k[\GL_{2(1)}],\rho_{\rm reg})\xrightarrow{\simeq}(\pi_*\calO_P,\rho_{0})\otimes (k[\GL_{2(1)}],\rho_{\rm reg}).
\end{equation}
If one write
$k[\GL_{2(1)}]=k[z_{ij}]/(z_{ij}^2-\delta_{ij})$, then
\begin{equation*}
\begin{pmatrix}
\iota(z_{11}) & \iota(z_{12})\\
\iota(z_{21}) & \iota(z_{22})
\end{pmatrix}
=
\begin{pmatrix}
z_{11} & z_{12}\\
z_{21} & z_{22}
\end{pmatrix}^{-1}.
\end{equation*}
Therefore, via the map (\ref{eq:G-equiv}), the element
\begin{equation*}
e_{ij}\Def\sum_{q=1}^2 y_{iq}\otimes z_{qj}\in\pi_*\calO_P\otimes k[\GL_{2(1)}]
\end{equation*}
is mapped to $y_{ij}\otimes 1$, which belongs to 
\begin{equation*}
\bigl((\pi_*\calO_P,\rho_{0})\otimes (k[\GL_{2(1)}],\rho_{\rm reg})\bigl)^{\GL_{2(1)}}=k[x^{\pm 1}][ y_{ij}\otimes 1\,|\,1\le i, j\le 2].
\end{equation*}
Hence, by the equation (\ref{eq:G-equiv 2}), we find that
\begin{equation*}
E^{(0)}=k[x^{\pm 1}][e_{ij}\,|\,1\le i,j\le 2].
\end{equation*}
Notice that $E^{(i)}\simeq k[x^{\pm 1}]^{\oplus 2^4}$ with free basis $\{e_{11}^{m_{11}}e_{12}^{m_{12}}
e_{21}^{m_{21}}e_{22}^{m_{22}}\,|\,m_{ij}=0,1\}$.

Let $A\in\GL_{2^4}(k[x^{(-1)\pm 1}])$ be the representation matrix of the inverse isomorphism $\sigma^{(0)-1}$ with respect to the basis
\begin{equation*}
\begin{aligned}
& 1,\,\nu_{11}\nu_{21},\,\nu_{12}\nu_{22},\,\nu_{11}\nu_{12}\nu_{21}\nu_{22},;\\
&\nu_{11},\,\nu_{12},\,\nu_{21},\,\nu_{22},\\
&\nu_{11}\nu_{12},\,\nu_{11}\nu_{22},\,\nu_{12}\nu_{21},\nu_{21}\nu_{22};\\
&\nu_{11}\nu_{12}\nu_{21},\,\nu_{11}\nu_{12}\nu_{22},\,\nu_{11}\nu_{21}\nu_{22},\,\nu_{12}\nu_{21}\nu_{22}
\end{aligned}
\end{equation*}
for $\nu\in\{e,z\}$. One then obtain:
\begin{equation*}
A~=~
\begin{pmatrix}
C& O& O& O&\\
& B\otimes E & O&B\otimes D&\\
& & B\otimes B &O&\\
&O & & B\otimes x^{(-1)m} E
\end{pmatrix},
\end{equation*}
with
\begin{equation*}
\begin{aligned}
&B=
\begin{pmatrix}
f'_{11}(x^{(-1)})&f'_{21}(x^{(-1)})\\
f'_{12}(x^{(-1)})&f'_{22}(x^{(-1)})
\end{pmatrix}
;\quad E=
\begin{pmatrix}
1 & 0\\
0 & 1
\end{pmatrix}
;\\
&C=
\begin{pmatrix}
1& f'_{11}(x^{(-1)})f'_{21}(x^{(-1)}) & f'_{12}(x^{(-1)})f'_{22}(x^{(-1)})&*\\
0&x^{(-1)m} & 0 & f'_{12}(x^{(-1)})f'_{22}(x^{(-1)})x^{(-1)m}\\
0&0&x^{(-1)m}  & f'_{11}(x^{(-1)})f'_{21}(x^{(-1)})x^{(-1)m}\\
0&0&0&x^{m}
\end{pmatrix};\\
&D=
\begin{pmatrix}
0& f'_{12}(x^{(-1)})f'_{22}(x^{(-1)})\\
f'_{11}(x^{(-1)})f'_{21}(x^{(-1)})& 0
\end{pmatrix},
\end{aligned}
\end{equation*}
where $\{f'_{ij}(x^{(-1)})\}$ is defined by:
\begin{equation*}
f'_{ij}(x^{(-1)})^2=f_{ij}(x).
\end{equation*}
Here, notice that $x^{(-1)2}=x$ and that $\det B=x^{(-1)m}$. Furthermore, the assumption on $\{f_{ij}(x)\}$ implies the same condition on $\{f'_{ij}(x^{(-1)})\}$.

Now one can reduce the problem to calculating the right hand side of the following equation:
\begin{equation*}
\Hom_{\Str(\G_m,1)}(\unit,P^{\str})\simeq\bigl\{(\underline{a},\underline{b})\in k^{\oplus 2^4}\times k[x^{\pm 1}]^{\oplus 2^4}\,\bigl |\,A\cdot\underline{a}=\underline{b}\bigl\}.
\end{equation*}

First we consider the case where $2\mid m$. In this case, without loss of generality, we may assume that $m=0$. Note that the condition that $\det(f_{ij}(x))=1$ implies that
the set
\begin{equation*}
\{\overline{f_{11}f_{12}},\,\overline{f_{21}f_{22}}\}
\end{equation*}
is also linearly independent over $k$. Indeed, this follows from the equation:
\begin{equation*}
\begin{pmatrix}
f_{11}f_{12}\\
f_{21}f_{22}
\end{pmatrix}
=
\begin{pmatrix}
f^2_{12}&f^2_{11}\\
f^2_{22}&f^2_{21}
\end{pmatrix}
\begin{pmatrix}
f_{11}f_{21}\\
f_{12}f_{22}
\end{pmatrix}.
\end{equation*}
We will solve the simultaneous equations $A\cdot \underline{a}=\underline{b}$ from the bottom. Then the condition that the set $\{\overline{f'_{11}},\,\overline{f'_{21}}\}$, or $\{\overline{f'_{12}},\,\overline{f'_{22}}\}$ is linearly independent implies that
\begin{equation*}
a_i=0\quad (13\le i\le 16).
\end{equation*}
Moreover, the conditions that $\{\overline{f'_{11}f'_{12}},\,\overline{f'_{21}f'_{22}}\}$ is linearly independent and that $\det B=f'_{11}f'_{22}-f'_{12}f'_{21}=1$ imply that
\begin{equation*}
\begin{aligned}
&a_i=0\quad(i=9,12),\\
&a_{10}=a_{11}.
\end{aligned}
\end{equation*}
By solving the equations
\begin{equation*}
\begin{pmatrix}
C&O\\
O&B\otimes E
\end{pmatrix},
\end{equation*}
again combined with the assumption on $\overline{f'_{ij}}$, we have
\begin{equation*}
a_i=0\quad (2\le i\le 8).
\end{equation*}
All the above computations then imply
\begin{equation*}
\Hom_{\Str(\G_m,1)}(\unit,P^{\str})\simeq k\oplus k\cdot \det(z_{ij})\simeq k[\mu_2].
\end{equation*}

Finally, let us consider the case where $2\nmid m$. In this case, the equations in the part $B\otimes x^{(-1)m}E$ cannot imply $a_{i}=0~(13\le i\le 16)$. 
However, by solving the simultaneous equations
\begin{equation*}
\begin{pmatrix}
O&B\otimes E&O&B\otimes D\\
\end{pmatrix}\underline{a}=\underline{b},
\end{equation*}
we can conclude that
\begin{equation*}
a_i=0\quad (5\le i\le 8~\text{or $13\le i\le 16$}).
\end{equation*}
Next let us solve the part $C$. Notice that $\{\overline{x^{(-1)m}},\,\overline{f'_{11}f'_{21}x^{(-1)m}}\}$ is linearly independent over $k$. Thus we find that $a_4=0$. On the other hand, by the assumption that $\{\overline{f'_{11}f'_{21}},\,\overline{f'_{12}f'_{22}}\}$ is linearly independent, we also have $a_2=a_3=0$.
Finally let us solve the part $B\otimes B$. Then the condition that $\overline{f'_{11}f'_{21}}\neq 0$ in $k[x^{(-1)\pm 1}]/k[x^{(-1)\pm 2}]$ implies that $a_{10}=a_{11}$. Then the condition that $\{\overline{f_{11}f_{12}},\,\overline{f_{21}f_{22}},\,\overline{x^m}\}$ is linearly independent over $k$ implies that $a_{9}=a_{10}=a_{11}=a_{12}=0$. Therefore, we can conclude that
\begin{equation*}
\Hom_{\Str(\G_m,1)}(\unit,P^{\str})=k.
\end{equation*}
\end{proof}

As an immediate consequence of Theorem \ref{thm:main} (or its proof), we have:

\begin{cor}\label{cor:SL_2}
Assume $p=2$. Let
\begin{equation*}
\underline{f}=
\begin{pmatrix}
f_{22}(x)&f_{12}(x)\\
f_{21}(x)&f_{11}(x)
\end{pmatrix}
\in\SL_2(k[x])={\rm Mor}_k(\A^1_k,\SL_2)
\end{equation*}
be a $k$-morphism. Then the resulting $\SL_{2(1)}$-torsor $\underline{f}^*\SL_{2}^{(-1)}\to\A^1_k$ is saturated if and only if one of the following conditions is satisfied:
\begin{itemize}
\item $\Dim_k\langle\overline{f_{11}},\,\overline{f_{21}}\rangle=\Dim_k\langle\overline{f_{11}f_{21}},\,\overline{f_{12}f_{22}}\rangle=2$;
\item 
$\Dim_k\langle\overline{f_{12}},\,\overline{f_{22}}\rangle=\Dim_k\langle\overline{f_{11}f_{21}},\,\overline{f_{12}f_{22}}\rangle=2$.
\end{itemize}
\end{cor}

\begin{proof}
By considering the composition
\begin{equation*}
\G_m\hookrightarrow\A^1_k
\xrightarrow{\underline{f}}\SL_2\hookrightarrow\GL_2,
\end{equation*}
we are reduced to the case where $\underline{f}:\G_m\to\GL_2$ with $\det(\underline{f})=1$. Thus the statement follows from the argument in the proof of Theorem  \ref{thm:main}, noticing that the condition that
\begin{equation*}
\Dim_k\Hom_{\Str(\G_m,1)}(\unit,P^{\str})=2
\end{equation*}
implies the image of the homomorphism $\pi^{\loc}(\G_m)\to\GL_{2(1)}$ corresponding to the torsor $P=\underline{f}^*\GL_2^{(-1)}\to\G_m$ coincides with $\SL_{2(1)}$.\\
\end{proof}

\begin{rem}
For example, the embedding
\begin{equation*}
\A^1_k\ni a\mapsto
\begin{pmatrix}
1+a^2+a^3&a\\
a+a^2&1
\end{pmatrix}
\in\SL_2
\end{equation*}
satisfies the condition in Corollary \ref{cor:SL_2}.
\end{rem}

Combined with Lemma \ref{lem:any ht}, Theorem \ref{thm:main} also implies:

\begin{cor}\label{cor:GL_2}
Assume $p=2$. Then, there exists a surjective homomorphism
\begin{equation*}
\pi^{\loc}(\G_m)\twoheadrightarrow\varprojlim_{r>0}\GL_{2(r)}.
\end{equation*}
In particular, for any $r>0$, the $r$-th Frobenius kernel $\GL_{2(r)}$ appears as a finite quotient of $\pi^{\loc}(\G_m)$.
\end{cor}

\begin{proof}
It suffices to find a $k$-morphism $\underline{f}\in\GL_2(k[x^{\pm 1}])={\rm Mor}_k(\G_m,\GL_2)$ satisfying the condition of Theorem \ref{thm:main}. The morphism
\begin{equation*}
\G_m\ni a\mapsto
\begin{pmatrix}

\frac{1}{a}+\frac{1}{a^4}& \frac{1}{a^2}\\
1& \frac{1}{a}
\end{pmatrix}\in\GL_2
\end{equation*}
gives such one.
\end{proof}

\begin{rem}
(1) Corollary \ref{cor:SL_2}, combined with Lemma \ref{lem:any ht}, gives an answer to Question \ref{ques:main}(2) for the pair $(U,\Sigma)=(\A_k^1,\SL_2)$.\\
(2) Theorem \ref{thm:main} and the proof of Corollary \ref{cor:GL_2}, combined with Lemma \ref{lem:any ht},  gives an affirmative answer to Question \ref{ques:main}(1) for the pair $(U,\Sigma)=(\G_m,\GL_2)$ in the case where $k$ is of characteristic $p=2$.
Furthermore, since $\X(\GL_{2(r)})=\Z/2^r\Z\subset\Q_2/\Z_2$~(Note that the first equality follows from the exactness of $1\to\SL_{2(r)}\to\GL_{2(r)}\to\mu_{p^r}\to 1$ and $\X(\SL_{2(r)})=1$), Corollary \ref{cor:GL_2} gives an affirmative answer to Question  \ref{ques:purely insep} for $\G_m$ and for $\GL_{2(r)}~(r>0)$ in the case where $k$ is of characteristic $p=2$. However, we have restricted our attension to a special class of $k$-morphisms $\G_m\to\GL_2$, so it is not enough to give a complete  answer to Question \ref{ques:main}(2).
\end{rem}

\subsection{Bertini type theorem for finite local torsors and its application}

Finally let us prove a purely inseparable analogue of Bertini type theorem~(cf.~\cite{jo83}). As an application, we will give an affirmative answer to Question \ref{ques:main}(1)  for the pair $(\A_k^1,\Sigma)$ with $\Sigma$ a semi-simple simply connected algebraic group over $k$, whence, to Question \ref{ques:purely insep} for $\A_k^1$ and for $\Sigma_{(r)}~(r>0)$.

\begin{thm}\label{thm:bertini}
Let $n\ge 2$ be an integer. Let $k$ be a perfect field of positive characteristic $p>0$. Let $G$ be a finite local $k$-group scheme of height one and $\pi:P\to\A^n_k$ be a saturated $G$-torsor. Then there exists a closed immersion $\iota:\A^{n-1}_k\hookrightarrow\A^n_k$ such that the $G$-torsor $\iota^*P\to\A^{n-1}_k$ obtained by pulling back $P$ along $\iota$ is saturated as well.
\end{thm}

To prove this theorem, we will rely on the Tannakian  interpretation~(cf.~Theorem \ref{thm:esnault-hogadi}) again. Then we are reduced to showing the following lemma of linear algebras:

\begin{lem}\label{lem:bertini}
Let $n\ge 2$ be an integer. Let $k$ be a perfect field of positive characteristic $p>0$. Let $V_1,\dots,V_m\subset k[x_1,\dots,x_n]/k[x_1^p,\dots,x_n^p]$ be finite dimensional subspaces. Then there exists a polynomial $g=g(x_1,\dots,x_{n-1})$ so that the $k$-linear map
\begin{equation*}
\begin{aligned}
k[x_1,\dots,x_n]/k[x_1^p,\dots,x_n^p]&\to k[x_1,\dots,x_{n-1}]/k[x_1^p,\dots,x_{n-1}^p];\\
x_i\quad&\mapsto \quad 
\begin{cases}
\quad\quad x_i, & 1\le i\le n-1,\\
g(x_1,\dots,x_{n-1}), & i=n
\end{cases}
\end{aligned}
\end{equation*}
maps all the subspaces $V_i$ injectively into $k[x_1,\dots,x_{n-1}]/k[x_1^p,\dots,x_{n-1}^p]$.
\end{lem}

\begin{proof}
By considerling $V:=V_1+\dots+V_m$, we are reduced to the case where $m=1$. Let $V$ be a finite dimensional subspace of $k[x_1,\dots,x_n]/k[x_1^p,\dots,x_n^p]$. Without loss of generality, we may assume that $V$ is of the form
\begin{equation*}
V=\langle x_1^{m_1}\cdots x_n^{m_n}\,|\,0\le m_i\le d;\, p\nmid m_i~\text{for some $i$}\rangle
\end{equation*}
for a sufficiently large integer $d>0$. Take an integer $M>d$ with $p\nmid M$ and an integer $N>0$ with $p^N>d(M+1)$. Let us define
\begin{equation*}
g\Def\prod_{i=1}^{n-1}(x_i^{p^N}+x_i^M).
\end{equation*}
Note that the following subset of $k[x_1,\dots,x_{n-1}]/k[x_1^p,\dots,x_{n-1}^p]$ is linearly independent over $k$:
\begin{equation*}
\begin{aligned}
&x_1^{m_1}\cdots x_{n-1}^{m_{n-1}}~&&(0\le m_i\le d, m_n=0;~p\nmid m_i~\text{for some $1\le i\le n-1$}),\\
& \prod_{i=1}^{n-1}x_i^{m_i+m_nM}~&&(p\mid m_i~\text{for any $1\le i\le n-1$, whence $p\nmid m_n$}),\\
&\prod_{i=1}^{n-1}x_i^{m_i+m_np^N}&&
(0\le m_i\le d;~m_n\neq 0;~p\nmid  m_i~\text{for some $1\le i\le n-1$}).
\end{aligned}
\end{equation*}
One can then conclude that the polynomial $g$ gives a desired polynomial. 
\end{proof}

\begin{proof}[Proof of Theorem \ref{thm:bertini}]
We will show that there exists a closed immersion $\iota:\A^{n-1}_k\hookrightarrow\A^n_k$ such that
\begin{equation*}
\Dim_k\Hom_{\Str(\A^{n-1}_k,1)}(\unit,(\iota^*P)^{\str})=1.
\end{equation*}
Note that $(\iota^*P)^{\str}=\iota^*(P^{\str})$ and that, from the assumption, we have
\begin{equation}\label{eq:bertini assump}
\Dim_k\Hom_{\Str(\A^n_k,1)}(\unit,P^{\str})=1.
\end{equation}
Let $P^{\str}=\{E^{(i)},\sigma^{(i)}\}$.
By Serre's conjecture on vector bundles on an affine space~(cf.~\cite{la78}), the vector bundle $E^{(0)}$ is a free $\calO_{\A^n_k}$-module. Let
\begin{equation*}
A\in\GL_{q}(k[x_1^{(-1)},\dots,x_n^{(-1)}])
\end{equation*}
with $q=\Dim_{k}k[G]$, which is some  power of $p$, be the representation matrix of $\sigma^{(0)-1}$ with respect to some basis. The assumption (\ref{eq:bertini assump}) then amounts to saying that the vector space
\begin{equation*}
\Hom_{\Str(\A_k^n,1)}(\unit, P^{\str})\simeq\{\underline{a}\in k^{\oplus q}\,|\,A\cdot\underline{a}\in k[x_1,\dots,x_n]^{\oplus q}\}
\end{equation*}
is of dimension one. Then, by permuting basis if necessary, we may assume that
\begin{equation*}
\{\underline{a}\in k^{\oplus q}\,|\,A\cdot\underline{a}\in k[x_1,\dots,x_n]^{\oplus q}\}
=\biggl\{
\begin{pmatrix}
a\\
0\\
\vdots\\
0
\end{pmatrix}
\biggl |a\in k\biggl\}.
\end{equation*}
Then, the matrix $A$ has the following form:
\begin{equation*}
A=
\begin{pmatrix}
    a_{11} &  *   &   \cdots   & *    \\
     a_{21}    & * &  \cdots     & *    \\
    \vdots &     & \ddots &     \\
     a_{q1}  &   *  &  \cdots      & *
\end{pmatrix}
\quad\text{with $a_{i1}\in k[x_1,\dots,x_n]~(1\le i\le q)$}.
\end{equation*}
Let us write
\begin{equation*}
B\Def
\begin{pmatrix}
a_{12}&\cdots &a_{1q}\\
&\cdots &\\
a_{q2}&\cdots &a_{qq}
\end{pmatrix}.
\end{equation*}
Then the condition that $\Dim_k\Hom_{\Str(\A_k^1,1)}(\unit,P^{\str})=1$ is equivalent to the condition that
\begin{equation*}
B\cdot\underline{b}\not\in k[x_1,\dots,x_n]^{\oplus q}~(0\neq \underline{b}\in k^{\oplus q-1}).
\end{equation*}
Therefore, we are reduced to showing the following statement: if a matrix $B=(b_{ij})\in {\rm Matrix}_{s,t}(k[x_1^{(-1)},\dots,x^{(-1)}_n])$ satisfies the condition 
\begin{equation*}
B\cdot\underline{b}\not\in k[x_1,\dots,x_n]^{\oplus t}~(0\neq\underline{b}\in k^{\oplus s}),
\end{equation*}
then there exists a polynomial $g=g(x_1^{(-1)},\dots,x^{(-1)}_{n-1})$ such that the condition
\begin{equation*}
B'\cdot\underline{b}\not\in k[x_1,\dots,x_{n-1}]^{\oplus t}~(0\neq\underline{b}\in k^{\oplus s}),
\end{equation*}
is fulfilled. Here $B'=(b'_{ij})$ is a matrix with $b'_{ij}=b_{ij}(x_1^{(-1)},\dots,x_{n-1}^{(-1)},g)$. Indeed, by applying Lemma \ref{lem:bertini} to the subspaces
\begin{equation*}
V_i\Def\langle\overline{b_{i1}},\dots,\overline{b_{is}}\rangle\subset k[x_1^{(-1)},\dots,x_n^{(-1)}]/k[x_{1},\dots,x_n]~(1\le i\le t),
\end{equation*}
we can find a polynomial $g=g(x_1^{(-1)},\dots,x_{n-1}^{(-1)})$ so that the homomorphism
\begin{equation*}
k[x_1^{(-1)},\dots,x_n^{(-1)}]/k[x_1,\dots,x_n]\to k[x_1^{(-1)},\dots,x_{n-1}^{(-1)}]/k[x_1,\dots,x_{n-1}]
\end{equation*}
induced by $x_n^{(-1)}\mapsto g(x_1^{(-1)},\dots,x_{n-1}^{(-1)})$ maps all the subspaces $V_i$ injectively into
\begin{equation*}
k[x_1^{(-1)},\dots,x_{n-1}^{(-1)}]/k[x_1,\dots,x_{n-1}].
\end{equation*}
Then for any $\underline{b}\in k^{\oplus s}$, the condition that $B'\cdot\underline{b}\in k[x_1,\dots,x_{n-1}]^{\oplus t}$ implies the condition that $B\cdot\underline{b}\in k[x_1,\dots,x_{n}]^{\oplus t}$. This completes the proof.
\end{proof}

\begin{cor}\label{cor:bertini}
Let $k$ be a perfect field of positive characteristic $p>0$. Let $\Sigma$ be a semi-simple simply connected algebraic group over $k$. Then there exists a surjective homomorphism
\begin{equation*}
\pi^{\loc}(\A^1_k)\twoheadrightarrow
\varprojlim_{r>0}\Sigma_{(r)}.
\end{equation*}
In particular, for any $r>0$, the $r$-th Frobenius kernel $\Sigma_{(r)}$ of $\Sigma$ appears as a finite quotient of $\pi^{\loc}(\A_k^1)$.
\end{cor}

\begin{proof}
We will adopt the argument of \cite[Section 3.2]{se92}.
Let us find a $k$-morphism $\A^1_k\to\Sigma$ so that the resulting tower of torsors
\begin{equation*}
\dots\to \Sigma^{(-r)}|_{\A^1_k}\to\dots\to \Sigma^{(-1)}|_{\A^1_k}\to \A^1_k
\end{equation*}
is saturated.
By the virtue of Lemma \ref{lem:any ht}, it suffices to find a $k$-morphism $\A^1_k\to\Sigma$ so that $\Sigma^{(-1)}|_{\A^1_k}$ is saturated.
Fix a maximal torus $T<\Sigma$ and a set of positive roots. Let $U^{+}<\Sigma$ (resp. $U^-<\Sigma$) be the unipotent radical corresponding to the positive roots (resp. the negative roots). Let us define a $k$-morphism $f:U^{+}\times U^{-}\to\Sigma$ by the composition
\begin{equation*}
U^{+}\times U^{-}\hookrightarrow\Sigma\times\Sigma\xrightarrow{m}\Sigma.
\end{equation*}
Here the first map is the natural inclusion and the second one the multiplication of $\Sigma$. Then the pulling-back $P\Def f^*\Sigma$ defines a saturated $\Sigma_{(1)}$-torsor over $U^{+}\times U^{-}$. Indeed, notice that the diagram
\begin{equation*}
\begin{xy}
\xymatrix{
U^{\pm (-1)}\ar[d]_{F^{(1)}}\ar @{^{(}-{>}} [rr] && \Sigma^{(-1)}\ar[d]^{F^{(1)}}\\
U^{\pm}\ar @{^{(}-{>}} [r]^{i_{\pm}\quad}&U^{+}\times U^{-}\ar[r]^{f}&\Sigma
}
\end{xy}
\end{equation*}
commutes. Then the $\Sigma_{(1)}$-torsor $i_{\pm}^*P\to U^{\pm}$ is reduced to the saturated $U^{\pm}_{(1)}$-torsor $U^{\pm (-1)}\to U^{\pm}$, i.e., 
\begin{equation*}
i_{\pm}^*P\simeq {\rm Ind}^{\Sigma_{(1)}}_{U^{\pm}_{(1)}}(U^{\pm (-1)})\Def U^{\pm (-1)}\times^{U^{\pm}_{(1)}}\Sigma_{(1)}.
\end{equation*}
If we denote by $\phi:\pi^{\loc}(U^{+}\times U^{-})\to \Sigma_{(1)}$ (resp. $\psi^{\pm}:\pi^{\loc}(U^{\pm})\to U^{\pm}_{(1)}$) the homomorphism corresponding to the torsor $P$ (resp. $i_{\pm}^*P$), this amounts to saying that the diagram homomorphism
\begin{equation*}
\begin{xy}
\xymatrix{
\pi^{\loc}(U^{\pm})\ar[r]
\ar @{-{>>}} [d]_{\psi_{\pm}}&\pi^{\loc}(U^{+}\times U^{-})\ar[d]^{\phi}\\
U^{\pm}_{(1)}\ar @{^{(}-{>}} [r]&\Sigma_{(1)}
}
\end{xy}
\end{equation*}
commutes. Therefore, ${\rm Im}(\phi)\supset U^{\pm}_{(1)}$, whence ${\rm Im}(\phi)=\langle U^{\pm}_{(1)}\rangle=\Sigma_{(1)}$, where, for the last equality, see Remark \ref{rem:purely insep}(4).  Therefore, $P$ is a saturated $\Sigma_{(1)}$-torsor over $U^{+}\times U^{-}\simeq\A^N_k$ for some $N>1$. Then by applying Theorem \ref{thm:bertini}, we can conclude that there exists a closed immersion $\iota:\A^1_k\hookrightarrow U^+\times U^-$ so that $\iota^*P\to\A^1_k$ is saturated as well. Therefore, the $k$-morphism $f\circ\iota:\A^1_k\to\Sigma$ gives a desired one. This completes the proof.
\end{proof}

%%%%%%   Reference   %%%%%%%%%%%%%%%%%%%


\small



\begin{thebibliography}{99}
\bibitem{ab57}
{Abhyankar, S.}, 
\textit{Coverings of algebraic curves}, 
{Amer. J. Math.}, Vol.~{79} (1957),~{825--856}.
\bibitem{ab92}
{Abhyankar, S.}, 
\textit{Galois theory on the line in nonzero characteristic}, 
{Bull. Amer. Math. Soc. (N.S.)}, Vol.~{27},~No. 1~(1992),~68--133.
\bibitem{aov08}
Abramovich, D., Olsson, M., Vistoli, A., 
\textit{Tame stacks in positive characteristic}, 
{Ann. Inst. Fourier (Grenoble)}, Vol.~58,~No.~4~(2008),~{1057--1091}.
\bibitem{an11}
Antei, M., 
\textit{On the abelian fundamental group scheme of a family of varieties}, 
Israel J. Math., Vol.~186~(2011),~{427--446}.\bibitem{bo09}
Borne, N., 
\textit{Sur les repr\'esentations du groupe fondamental d'une vari\'et\'e priv\'ee d'un diviseur \`a croisements normaux simples}, 
{Indiana Univ. Math. J.}, Vol.~58,~No.~1~(2009),~137--180.
\bibitem{bv15}
Borne, N., Vistoli, A., 
\textit{The Nori fundamental gerbe of a fibered category}, 
{J. Algebraic Geom.}, Vol.~24,~No.~2~(2015),~{311--353}.
\bibitem{bouw00}
Bouw, I. I., 
\textit{The {$p$}-rank of curves and covers of curves}, 
Courbes semi-stables et groupe fondamental en g\'eom\'etrie alg\'ebrique ({L}uminy, 1998), Progr. Math., Birkh\"auser, Basel, Vol.~187, (2000),~{267--277}.
\bibitem{cept96}
Chinburg, T., Erez, B., Pappas, G., Taylor, M. J.,
\textit{Tame actions of group schemes: integrals and slices}, 
{Duke Math. J.}, Vol.~82,~No.~2~(1996),~{269--308}.
\bibitem{dm82}
Deligne, P., Milne, J., 
\textit{Tannakian Categories},
Lectures Notes in Mathematics 900, Springer-Verlag, Berlin-New York, 1982.
\bibitem{de90}
Deligne, P., 
\textit{Cat\'egories tannakiennes}, 
{The {G}rothendieck {F}estschrift, {V}ol.\ {II}},
{Progr. Math.}, {Birkh\"auser Boston, Boston, MA}, Vol.~87~(1990),~{111--195}.
\bibitem{ds07}
dos Santos, J. P. P., 
\textit{Fundamental group schemes for stratified sheaves}, 
{J. Algebra}, Vol.~317, No.~2~(2007),~{691--713}.
\bibitem{ehs08}
{Esnault, H., Hai, P. H., Sun, X.},
\textit{On {N}ori's fundamental group scheme}, 
{Geometry and dynamics of groups and spaces}, 
{Progr. Math.}, {Birkh\"auser, Basel}, Vol. {265}~(2008),~{377--398}.
\bibitem{eh12}
{Esnault, H., Hogadi, A.},
\textit{On the algebraic fundamental group of smooth varieties in characteristic {$p>0$}}, 
{Trans. Amer. Math. Soc.},~Vol. 364,~No.~5~(2012),~{2429--2442}.
\bibitem{fj08}
{Fried, M. D., Jarden, M.}, \textit{Field arithmetic}, 
{Ergebnisse der Mathematik und ihrer Grenzgebiete. 3. Folge. A Series of Modern Surveys in Mathematics}, 
Vol. {11}, {Third Edition}, {Springer-Verlag, Berlin}, (2008).
\bibitem{gr71}
Grothendieck, A.,
\textit{Rev\^etements \'etales et groupe fondamental},
SGA1, Lecture Notes in Mathematics 224, Springer-Verlag, Berlin-New York, 1971.
\bibitem{gi75}
Gieseker, D., 
\textit{Flat vector bundles and the fundamental group in non-zero characteristics}, 
{Ann. Scuola Norm. Sup. Pisa Cl. Sci. (4)}, Vol.~2, No.~1~(1975),~{1--31}.
\bibitem{gi12}
Gillibert, J., 
\textit{Tame stacks and log flat torsors}, 
arXiv:1203.6870v1. 
\bibitem{gi71}
Giraud, J., 
\textit{Cohomologie non ab\'elienne}, 
{Die Grundlehren der mathematischen Wissenschaften, Band 179}, 
{Springer-Verlag, Berlin-New York}, 1971.
\bibitem{ha94}
Harbater, D., 
\textit{Abhyankar's conjecture on {G}alois groups over curves}, 
{Invent. Math.}, Vol.~{117}, No.~1~(1994),~{1--25}.
\bibitem{ha14}
Harbater, D., Obus, A., Pries, R., Stevenson, K., 
\textit{Abhyankar's Conjectures in Galois Theory: Current Status and Future Directions}, 
arXiv: 1408.0859v1.
\bibitem{ja03}
Jantzen, J. C., 
\textit{Representations of algebraic groups}, 
{Mathematical Surveys and Monographs}, Vol.~{107},  
{Second Edition}, 
{American Mathematical Society, Providence, RI}, 
(2003).
\bibitem{jo83}
Jouanolou, J.-P., 
\textit{Th\'eor\`emes de {B}ertini et applications}, 
{Progress in Mathematics}, Vol.~{42}, 
{Birkh\"auser Boston, Inc., Boston, MA}, (1983).
\bibitem{ka85}
Kambayashi, T., 
\textit{Nori's construction of {G}alois coverings in positive characteristics}, {Algebraic and topological theories ({K}inosaki, 1984)}, {Kinokuniya, Tokyo} (1986),~{640--647}.
\bibitem{la78}
Lam, T. Y., \textit{Serre's conjecture}, 
{Lecture Notes in Mathematics, Vol. 635}, {Springer-Verlag, Berlin-New York}, (1978).
\bibitem{ma12}
Marques, S., 
\textit{Actions mod\'er\'ees de sch\'emas en groupes affines et champs mod\'er\'es}, 
{C. R. Math. Acad. Sci. Paris}, Vol.~350,~No.~3-4~(2012)~125--128.
\bibitem{mu08}
Mumford, D., 
\textit{Abelian varieties}, 
Tata Institute of Fundamental Research Studies in Mathematics, 
Vol.~5~(2008).
\bibitem{no76}
Nori, M. V.,
\textit{On the representations of the fundamental group},
Compositio Math.~Vol.~33,~No.~1~(1976),~29--41.
\bibitem{no82}
Nori, M. V.,
\textit{The fundamental group-scheme},
Proc. Indian Acad. Sci. Math. Sci.,~Vol.~91,~No.~2~(1982),~73--122.
\bibitem{no94}
Nori, M. V., 
\textit{Unramified coverings of the affine line in positive characteristic}, 
{Algebraic geometry and its applications ({W}est {L}afayette, 
{IN}, 1990)}, Springer, New York (1994),~{209--212}.
\bibitem{ps00}
{Pacheco, A., Stevenson, K. F.}, 
\textit{Finite quotients of the algebraic fundamental group of projective curves in positive characteristic}, 
{Pacific J. Math.},~Vol.~192,~No.~1~(2000),~143--158.
\bibitem{ra94}
Raynaud, M., 
\textit{Rev\^etements de la droite affine en caract\'eristique {$p>0$} et conjecture d'{A}bhyankar}, {Invent. Math.}, 
Vol.~{116}, No.~1-3~(1994),~{425--462}.
\bibitem{se90}
Serre, J.-P., 
\textit{Construction de rev\^etements \'etales de la droite affine en caract\'eristique {$p$}}, 
{C. R. Acad. Sci. Paris S\'er. I Math.}, Vol.~311,~No.~6~(1990),~{341--346}.
\bibitem{se92}
Serre, J.-P.,~
\textit{Rev\^etements de courbes alg\'ebriques}, 
{S\'eminaire Bourbaki, Vol. 1991/92}, {Ast\'erisque},
No.~{206}~(1992),~{Exp.\ No.\ 749, 3, 167--182}.
\bibitem{stack}
The Stacks Project Authors, 
\textit{Stacks Project}, https://stacks.math.columbia.edu.
\bibitem{ta99}
Tamagawa, A., 
\textit{On the fundamental groups of curves over algebraically closed fields of characteristic {$>0$}}, 
{Internat. Math. Res. Notices}, No.~16 (1999),~{853--873}.
\bibitem{ta04}
Tamagawa, A., 
\textit{Finiteness of isomorphism classes of curves in positive  characteristic with prescribed fundamental groups}, 
{J. Algebraic Geom.}, Vol.~13, No.~4 (2004),~{675--724}.
\bibitem{un10}
{\"Unver, S.}, 
\textit{On the local unipotent fundamental group scheme}, 
{Canad. Math. Bull.}, Vol.~53,~No.~1~(2010),~{187--191}.
\bibitem{vi10}
Viviani, F., 
\textit{Simple finite group schemes and their infinitesimal deformations}, 
{Rend. Semin. Mat. Univ. Politec. Torino}, 
Vol.~{68}, No.~2~(2010),~{171--182}.
\bibitem{wa79}
Waterhouse, W. C.,
\textit{Introduction to affine group schemes},
Graduate Texts in Mathematics 66, Springer-Verlag, New York-Berlin, 1979.
\bibitem{za16}
Zalamansky, G., 
\textit{Ramification of inseparable coverings of schemes and application to diagonalizable group actions}, 
arXiv:1603.09284v1.
\bibitem{zh13}
Zhang, L., 
\textit{The homotopy sequence of {N}ori's fundamental group},  {J. Algebra}, Vol.~{393}~(2013),~{79--91}.
\end{thebibliography}
\end{document}